\documentclass[11pt,reqno]{amsart}
\usepackage{setspace}
\usepackage{footnote}
\usepackage{graphicx}
\usepackage{xcolor}
\usepackage{cmap,mathtools}
\usepackage[T1]{fontenc}
\usepackage[utf8]{inputenc}
\usepackage[english]{babel}
\usepackage{amsfonts,amssymb,amsthm,amsmath,enumerate,bm}
\usepackage{url}
\usepackage[shortlabels]{enumitem}
\usepackage{float}
\usepackage{secdot}

\allowdisplaybreaks

\usepackage{graphicx}
\usepackage{epstopdf}

\usepackage{a4wide}
\usepackage{xparse}
\usepackage{xcolor}
\usepackage[colorlinks=true,linktocpage,pdfpagelabels,bookmarksnumbered,bookmarksopen]{hyperref}
\definecolor{ForestGreen}{rgb}{0.1,0.6,0.05}
\definecolor{EgyptBlue}{rgb}{0.063,0.1,0.6}
\hypersetup{
	colorlinks=true,
	linkcolor=EgyptBlue,         
	citecolor=ForestGreen,
	urlcolor=olive
}

\usepackage[hyperpageref]{backref}

\newtheorem{theorem}{Theorem}

\newtheorem{proposition}[theorem]{Proposition}
\newtheorem{lemma}[theorem]{Lemma}
\newtheorem{corollary}[theorem]{Corollary}
\theoremstyle{definition}
\newtheorem{definition}[theorem]{Definition}
 \newtheorem{rmk}[theorem]{Remark}
\newtheorem{example}[theorem]{Example}

\let\OLDthebibliography\thebibliography
\renewcommand\thebibliography[1]{
	\OLDthebibliography{#1}
	\setlength{\parskip}{1pt}
	\setlength{\itemsep}{1pt plus 0.3ex}
}

\numberwithin{equation}{section}
\numberwithin{theorem}{section}
\numberwithin{equation}{section}
\numberwithin{theorem}{section}

\DeclarePairedDelimiter\abs{\lvert}{\rvert}%
\DeclarePairedDelimiter\norm{\lVert}{\rVert}%
\makeatletter
\let\oldnorm\norm
\def\norm{\@ifstar{\oldnorm}{\oldnorm*}}

\newcommand{\al} {\alpha}

\newcommand{\De} {\Delta}

\newcommand{\Ga} {\Gamma}
\newcommand{\om} {\omega}
\newcommand{\Om} {\Omega}

\newcommand{\la} {\lambda}

\newcommand{\Gr} {\nabla}

\newcommand{\noi} {\noindent}

\newcommand{\ra} {\rightarrow}

\newcommand{\wra} {\rightharpoonup}
\newcommand{\var} {\varepsilon}

 \def\Dp{{{\mathcal D}_p(\Om)}}
 \def\Dph{{{\mathcal D}_p(\Om)^H}}

\newcommand{\wrastar} {\overset{\ast}{\rightharpoonup}}

\newcommand\restr[2]{{
  \left.\kern-\nulldelimiterspace 
  #1 
  \right|_{#2} 
  }}

\usepackage[hyperpageref]{backref}

\def\A{{\mathcal A}}
\def\C{{\mathcal C}}
\def\D{{\mathcal D}}
\def\E{{\mathcal E}}

\def\N{{\mathbb N}}
\def\F{{\mathcal F}}

\def\cset{{\subset \subset }}
\def\R{{\mathbb R}}

\def\({{\Big(}}
\def\){{\Big)}}

\def\ws2{{\F_{\frac{N}{2}}}}

\def\c1{{\C_c^1}}

\def\d{{\rm d}}
\def\dr{{\rm d}r}

\def\dz{{\rm d}z}
\def\dx{{\rm d}x}

\def\H{{\mathcal{H}}}

\usepackage[foot]{amsaddr}
\usepackage[pagewise]{lineno}\nolinenumbers

\makeatother

\date{}
\begin{document}
\title{On the generalised Br\'ezis-Nirenberg problem}
\author{T. V. Anoop$^{1}$}
\address{$^{1}$ Department of Mathematics, Indian Institute of Technology, Madras, India.}
\email{anoop@\text{iitm}.ac.in}

\author{Ujjal Das$^{2^*}$}
\address{$^2$ corresponding author, Department of Mathematics, Technion - Israel Institute of
		Technology,   Haifa, Israel.}
\email{ujjaldas@campus.technion.ac.il, ujjal.rupam.das@gmail.com}

\maketitle

 \begin{abstract}
For $ p \in (1,N)$ and a domain $\Omega$ in $\mathbb{R}^N$, we study the following quasi-linear problem involving the critical growth:
\begin{eqnarray*} 
 -\Delta_p u - \mu g|u|^{p-2}u  =  |u|^{p^{*}-2}u \  \mbox{ in } \mathcal{D}_p(\Omega), 
\end{eqnarray*}
where $\Delta_p$ is the $p$-Laplace operator defined as $\Delta_p(u) = \text{div}(\abs{\nabla u}^{p-2} \Gr u),$ $p^{*}= \frac{Np}{N-p}$ is the critical Sobolev exponent and $\mathcal{D}_p(\Omega)$ is the Beppo-Levi space defined as the  completion of 
$\text{C}_c^{\infty}(\Omega)$ with respect to the norm $\|u\|_{\mathcal{D}_p} := \left[  \displaystyle \int_{\Omega} |\nabla u|^p \dx \right]^ \frac{1}{p}.$ In this article, we provide various sufficient conditions on $g$ and $\Omega$ so that the above problem admits a positive solution for certain range of $\mu$. As a consequence, for $N \geq p^2$, if $g $ is such that $g^+ \neq 0$ and the map $u \mapsto \displaystyle \int_{\Omega} |g||u|^p \dx$ is compact on $\mathcal{D}_p(\Omega)$, we show that the problem under consideration has a positive solution for certain range of $\mu$. Further, for $\Omega =\mathbb{R}^N$, we give a necessary condition for the existence of positive solution.
 \end{abstract} 

\medskip
\noindent
{\bf Mathematics Subject Classification (2020):} 35B33, 35J60, 58E30  \\
\noindent
{\bf Keywords:} Br\'ezis-Nirenberg type problem; Critical Sobolev exponent;  Concentration compactness; Principle of symmetric criticality; Pohozave type identity.
\maketitle
\section{Introduction}

For $p \in (1,N)$ and a domain $\Om$ in $\R^N$, the Beppo-Levi space $\mathcal{D}_p(\Om)$ is defined as the  completion of 
$\text{C}_c^{\infty}(\Om)$ with respect to the norm $\norm{u}_{\D_p} := \left[  \displaystyle \int_{\Omega} |\nabla u|^p \dx \right]^ \frac{1}{p}.$ 
In this article, we study the following quasi-linear partial differential equation involving the critical growth:
\begin{eqnarray} \label{Critical}
 -\De_p u - \mu g|u|^{p-2}u & = & |u|^{p^{*}-2}u \  \mbox{ in } \Dp, 
\end{eqnarray}
where $\Delta_p$ is the $p$-Laplace operator defined as $\Delta_p u = \text{div}(\abs{\Gr u}^{p-2} \Gr u),$ and $p^{*}= \frac{Np}{N-p}$ is the critical Sobolev exponent. Since the seminal work of Br\'ezis-Nirenberg \cite{Nirenberg}, many different classes of elliptic boundary value 
problems involving the critical exponent have been explored. 
For example, see \cite{Antonio,Miyagaki,Capozzi,Monica,Schechter,Tobias} for Laplacian ($p=2$) and \cite{AlonsoEstimates,Mercuri,Guedda,Alonso,Noussair,Chabrowski3,Cao} for 
p-Laplacian ($p \in (1,N)$). 
It is well known that, if $g \equiv 0$ in \eqref{Critical} i.e. the problem
\begin{eqnarray} \label{Critical0}
 -\De_p u & = & |u|^{p^{*}-2}u \  \mbox{ in } \Dp, 
\end{eqnarray}
does not admit any positive solution on a bounded, star-shaped domain \cite[Remark 1.2]{Nirenberg}. 
However, if $g \equiv 1$ on a bounded domain $\Om$, then \eqref{Critical} admits a positive solution for certain range (depending on $p$ and $N$) of $ \mu $, see \cite[for Laplacian]{Nirenberg} and \cite[for p-Laplacian]{AlonsoEstimates}. This suggests that certain perturbations  of \eqref{Critical0} of the form \eqref{Critical}  may admit a positive solution.
Also, in \cite{Shoyeb}, authors multiplied a weight function  $g$ to the right-hand side of \eqref{Critical0} and studied the existence of positive solutions.
In this article, we are interested in identifying a general class of $g$ so that \eqref{Critical} admits a positive solution. 
Many results have been appeared in this direction.
For example, for bounded domain $\Om$, $g$ is positive constant \cite{Nirenberg,AlonsoEstimates}, $g$ is bounded non-negative function \cite{Guedda, Egneel}, sign-changing $g \in \text{C}(\overline{\Om})$ \cite{Hsu}.
For $\Om=\R^N$,  Charles Swanson and Lao Sen \cite{Swanson}  considered nontrivial, non-negative $g$ in $ L^{\frac{N}{p}}(\R^N)$. For $p=2$, Chabrowski \cite{Chabrowski} considered sign-changing $g$ which is positive on a positive measure set and $g \in L^1(\R^N) \cap \text{C}(\R^N)$  with $  \lim_{|x|\ra \infty} \int_{B_l(x)}|g|^{\frac{r}{r-2}}\dx= 0$ for some $r\in (2,2^*),l>0$, and for
$p \in (1,N)$, Huang \cite{Huang} has taken  $g \in L^{\frac{p^*}{p^*-r}}(\R^N)$ for some $r \in (1,p^*)$. 
In \cite{Pavel}, Dr\'abek-Huang considered $g$ such that $g^+ \in L^{\infty}(\R^N) \cap L^{\frac{N}{p}}(\R^N)$ and $g^- \in L^{\infty}(\R^N)$.
In all the above mentioned results, the assumptions on $g$ ensure that the map $$G_p(u):=\displaystyle \int_{\Om} |g||u|^p \dx \, \, \text{on} \ \Dp$$
is compact (i.e. $G_p(u_n) \ra G_p(u)$ whenever $u_n \wra u$ in  $\Dp$). 
In this case, the non-compactness issues in dealing with \eqref{Critical} arise only  due to the presence of critical exponent or the unboundedness of the domain. Indeed, to tackle these kinds of non-compactness, one can use the celebrated concentration compactness principles of P. L. Lions \cite{Lions1a, Lions2a}. Moreover, in \cite{Cencelj}, authors considered multiple perturbations of \eqref{Critical0}, where the non-compactness of one perturbation is compensated by another perturbation.
The main objective of this article is to consider a single non-compact perturbation of \eqref{Critical0} by a potential $g$ and enlarge the class of $g$ that ensures the existence of a positive solution of \eqref{Critical}. We use a  concentration compactness principle that depends on $g$ (and in the later part, depends on a closed subgroup of $\mathcal{O}(N)$) that can simultaneously address all the aforementioned non-compactness (critical exponent, unboundedness of the domain, non-compact perturbations). 

Notice that, if $u \in \Dp$ is a solution to \eqref{Critical}, then
$$\int_{\Om} |\nabla u|^p \dx = \mu \int_{\Om} g |u|^p \dx + \int_{\Om} |u|^{p^*} \dx \geq  \mu \int_{\Om} g |u|^p \dx.$$
We priory assume that $g$ must satisfy the following Hardy type inequality: 
$$\int_{\Om} |g||u|^p \dx \leq C \int_{\Om} |\nabla u|^p \dx, \ \forall \,u \in \Dp,$$
for some $C>0.$ 
\begin{definition}
A function $g\in L^1_{loc}(\Om)$ satisfying the above inequality is called as a Hardy potential and  the space of all Hardy potentials is denoted by $\mathcal{H}_p(\Om).$
\end{definition} \noi One can define a norm on $\mathcal{H}_p(\Om)$ which makes it a Banach function space (see Section \ref{spaces}).  
For $g \in \mathcal{H}_p(\Om)$, we define
$$\mu_1(g,\Om)=\inf \left \{\int_{\Om} |\nabla u|^p \dx : \int_{\Om}|g||u|^p \dx=1, u\in \Dp\right \} \,.$$
If the underlying domain is unambiguous, then we simply write  $\mu_1(g)$ instead of $\mu_1(g,\Om).$ Observe that $\mu_1(g)>0$, and for each $ \mu \in (0,\mu_1(g))$,  $\|u\|_{\D_p,\mu}:= \displaystyle \left[\int_{\Om}[|\nabla u|^p - \mu g |u|^p] \dx \right]^{\frac{1}{p}}$ is a quasi-norm on $\D_p(\Om)$ and it is equivalent to $\|u\|_{\D_p}$. We also define
$$\mu_1(g,x)=\lim_{r \ra 0} \left[\inf \left \{\int_{\Om} |\nabla u|^p \dx : \int_{\Om}|g||u|^p \dx=1 : u\in \D_p(\Om \cap B_r(x))\right \} \right] \,; \, x \in \overline{\Om} \,,   $$
 and $$\Sigma_g=\{x \in \overline{\Om}: \mu_1(g,x)<\infty\} \,. $$

Now, for $g \in \mathcal{H}_p(\Om)$ and $\mu \in (0,\mu_1(g))$, we consider the functional 
$$J_{g,\mu}(u)= \displaystyle \int_{\Om} [|\nabla u|^p - \mu g |u|^p ] \dx $$
and define \begin{equation} \label{min}
 \E_{g,\mu}(\Om) = \inf \left \{ J_{g,\mu}(u) : u \in \mathbb{S}_p(\Om)  \right \} \,,
\end{equation}
where $\mathbb{S}_p(\Om):=\{u \in \Dp : \norm{u}_{p^*}=1\}.$
If $ \E_{g,\mu}(\Om)$ is attained at $v \in \mathbb{S}_p(\Om)$, then  the standard variational arguments ensure that $[J_{g,\mu}(v)]^{\frac{1}{p^*-p}}v$
is a non-trivial solution of \eqref{Critical}.
In order to investigate whether  $\E_{g,\mu}(\Om)$ is attained or not, consider a  minimizing sequence of $\E_{g,\mu}(\Om)$, say $(w_n)$ in $\Dp$.
It is not difficult to see that $(w_n)$ converges weakly to some $w\in \Dp$. By  zero extension, $w_n \in \mathcal{D}_p(\Om)$ can be considered as a $\mathcal{D}_p(\R^N)$ function whenever convenient. Using this convention, we define the measures $\nu_n$ as
\begin{eqnarray} \label{nuinfinity}
\nu_n(E) = \int_E |w_n-w|^{p^*} \dx, \text{ for every Borel set $E$ in } \R^N \nonumber \,,
 \end{eqnarray}
 and  the following quantity
 $$ \nu_{\infty} = \lim_{R \ra \infty} \overline{\lim_{n \ra \infty}} \int_{B_R^c} |w_n - w|^{p^*} \dx.$$
Since $(w_n)$  is bounded in $\Dp$, it follows that $(\nu_n)$ is bounded in the space of all regular, finite,  Borel signed-measures $\mathcal{M}(\R^N)$ with respect to the norm $\|\nu_n\|:=\nu_n(\R^N)$ (total variation of measures). By the Reisz representation theorem \cite[Theorem 14.14, Chapter 14]{Border}, $\mathcal{M}(\R^N)$ is the dual of $\text{C}_0(\R^N):=\overline{\text{C}_c(\R^N)}$ in $L^{\infty}(\R^N)$. Hence,  by the Banach-Aloglu theorem it follows that  $\nu_n \wrastar \nu$ in $\mathcal{M}(\R^N)$. 
It is important to mention that the measure $\nu$ together with the quantity $\nu_{\infty}$ captures the possible failure of strong convergence of $(w_n)$ in $L^{p^*}(\Om).$
Indeed, if $\nu=0=\nu_{\infty}$, then  $w_n \ra w$ in $L^{p^*}(\Om)$ (see Corollary \ref{Cor_ConCpct}), and consequently, $\E_{g,\mu}(\Om)$  is attained at $w$. If $\nu \neq 0 $ (or $\nu_{\infty}\neq 0$), then we say that $(w_n)$ concentrate on $\overline{\Om}$ (or at $\infty$).    
To study the concentration behaviour of $\nu$, we consider the following concentration function of $g$:  
\begin{align*}
\C_{g,\mu}(x):= \lim_{r \ra 0} \E_{g,\mu}(\Om \cap B_r(x)), \ \ \quad
\C_{g,\mu}(\infty):=  \lim_{R \ra \infty} \E_{g,\mu}(\Om \cap B_R^c)  \,. 
\end{align*}
We denote $\C_{g,\mu}^*(\Om) =  \displaystyle \inf_{\overline{\Om}} \C_{g,\mu}(x).$
Notice that 
\begin{align} \label{ineqqq}
     \E_{g,\mu}(\Om) \leq \C_{g,\mu}^* (\Om) \ \ \mbox{and} \ \  \E_{g,\mu}(\Om) \leq \C_{g,\mu}(\infty) \,.
\end{align}
\noi Later we see that the nature of the above inequalities (strict or not)
helps us to  determine whether $\nu$ is concentrated or not.
For the brevity, we make the following definition.
\begin{definition} \label{crisubcri}
Let $g \in \mathcal{H}_p(\Om)$ and $\mu \in (0, \mu_1(g))$. We say 
\begin{enumerate}[(i)]
 \item   $g$ is {\bf sub-critical in $\Om$} if $\E_{g,\mu}(\Om) < \C_{g,\mu}^*(\Om)$,  and {\bf sub-critical at infinity} if $\E_{g,\mu}(\Om) < \C_{g,\mu} (\infty)$, 
 \item  $g$ is  {\bf critical in $\Om$}
 if $\E_{g,\mu}(\Om) = \C_{g,\mu}^*(\Om)$,  and {\bf critical  at infinity}
 if $\E_{g,\mu}(\Om) = \C_{g,\mu} (\infty)$.
 \end{enumerate}
\end{definition}

\noi Let $$ \mathcal{F}_{p}(\Om):= \overline{\text{C}_c^{\infty}(\Om)}  \ \mbox{in} \ \mathcal{H}_p(\Om) \,. $$ In \cite{New}, it is proved that $\mathcal{F}_{p}(\Om)$ is the optimal space for $g$ such that $G_p$ is compact in $\Dp$. In this article, we assume that $g$ satisfies:
\begin{enumerate}[label={($\bf H1$)}]
\item $g \in \mathcal{H}_p(\Om), g^{-} \in \mathcal{F}_{p}(\Om) , \ \text{and} \  |\overline{{\Sigma}_g}|=0 \,.$ \label{H1}  
\end{enumerate}
Now we state our first result.
\begin{theorem} \label{mainthm}
Let $\Om$ be a domain in $\R^N$ and $g$ satisfies {\rm{\ref{H1}}}. If for some $\la \in (0,\mu_1(g)),$  $g$ is  sub-critical in $\Om$ and at infinity,  then \eqref{Critical} admits a positive solution for  $\mu=\la$.
\end{theorem}
\noi Our proof for the above theorem is based on the fact that, if $g$ is subcritical in $\Om$ and at infinity, then $\nu =0=\nu_{\infty}$ for any minimising sequence of $\E_{g,\mu}(\Om)$. In other words, the subcriticality (in $\Om$ and at $\infty$) of $g$ ensures that the minimising sequence does not concentrate in $\overline{\Om}$ as well as at $\infty$. 




Now, it is natural to ask whether the compact perturbations are sub-critical or not? The answer is  no, for example, $g \equiv0$ is a compact perturbation (as $G_p \equiv 0$ on $\Dp$) but not a sub-critical potential. On a bounded domain $\Om$,  $g \equiv 1$ is  a compact perturbation and  it is sub-critical for all $\mu \in (0,\mu_1(g))$ if $N\ge p^2$, see \cite[Theorem 1.2]{Nirenberg} (for $p=2$) and \cite[Theorem 4.2]{Clement} (for general $p$).  
Next we show that every   compact perturbation (with non-trivial positive part) behaves like $g \equiv 1$.
\begin{theorem} \label{perfect}
Let $N \geq p^2$ and $g \in \mathcal{F}_{p}(\Om)$ be such that $g^+\neq 0$. 
Then $g$ is {{sub-critical in $\Om$ and at the infinity}} for all $\mu \in (0,\mu_1(g))$.
\end{theorem}

\begin{rmk} \rm
 Indeed, the condition $N \geq p^2$ is crucial, for instance, if $p=2$ and $N=3$ then $g \equiv 1$  is not sub-critical for $\mu$ near $0.$ \cite[Corollary 1.1]{Nirenberg}.


\end{rmk}

Next we ask, what happens if  $g$ fails to be sub-critical either in $\Om$ or at infinity? Recall that, for the critical potential $g \equiv 0$, \eqref{Critical} does not admit any positive solution when $\Om$ is a bounded  star-shaped domain. However, it admits a positive solution when $\Om$ is an annular domain \cite[Section 4]{Kazdan} or entire $\R^N$ \cite{Lions2a}.
Similarly, if $\Om$ contains the origin, then $g(x)=\frac{1}{|x|^p}$ is critical in $\Om$ (see Example \ref{criatptsandinf}), and \eqref{Critical}  does not admit a positive solution for any $\mu \in \R$ \cite[Theorem 2.1]{Ghoussoub}, when $\Om$ is bounded star-shaped. On the other hand,  for the same $g$, \eqref{Critical}  does admit a  positive radial solution 
on entire $\R^N$ for every $\mu \in (0,\mu_1(g))$
\cite[Theorem 1.41]{Willem}. This indicates that, when the subcriticality fails, the symmetry  of the domain plays a vital role for the existence of solutions  of \eqref{Critical}. 
Here we study the non-subcritical cases under some additional symmetry assumptions on $\Om$ and $g$.
 
Consider the action of  a closed subgroup $H$ of $\mathcal{O}(N)$ (the group of all orthogonal matrices on $\R^N)$ on $\R^N$ given by $x\to h \cdot x,$ for $x\in \R^n$, $h\in H$, where $\cdot$ represents the matrix multiplication. We will be writing $h(x)$ instead of $h\cdot x$. For $x\in \R^N$, the orbit of $x$ is denoted by $Hx$ and for $E\subset \R^N$, the orbit of $E$ is denoted by $H(E)$. Thus, $Hx=\left \{h (x): h \in H \right \}$ and $H(E)=\{h(x): x \in E, h \in H\}$.
\begin{definition}
Let $\Om$ be a domain in $\R^N$ and $f:\Om \to \R.$ If  $H(\Om)=\Om,$ then we say $\Om$ is $H$-invariant. If $\Om$ is $H$-invariant and $f(h(x))=f(x),\forall\, h\in H$ (i.e., $f$ is constant on each  $H$ orbit), then we say $f$ is $H$-invariant.
\end{definition}
\noi For a  $H$-invariant domain $\Om$,  the $H$-action on $\Om$ naturally induces an action of $H$ on $\mathcal{D}_p(\Om)$ given by  $$\pi(h)(u)=u_h, \text{ where } u_h(x)=u(h^{-1} (x)).$$
Thus, $u\in \Dp$ is $H$-invariant if, and only if , $u_h=u.$
The set of all $H$-invariant functions in $\mathcal{D}_p(\Om)$ and $\mathbb{S}_p(\Om)$ are denoted by $\mathcal{D}_p(\Om)^H$ and $\mathbb{S}_p(\Om)^H$ respectively, i.e., 
\begin{align*}
 \mathcal{D}_p(\Om)^H &=\displaystyle\left\{u \in \mathcal{D}_p(\Om): u=u_h, \forall \,h \in H\right\} \\  \mathbb{S}_p(\Om)^H &=\left\{u \in \mathbb{S}_p(\Om): u=u_h, \forall \,h \in H \right\} \,. 
\end{align*}
\noi Clearly, for $H=\{Id_{\R^N}\}$, $\Dph=\Dp$ and $\mathbb{S}_p(\Om)^H=\mathbb{S}_p(\Om)$, and if $H=\mathcal{O}(N)$ and $\Om$ is a radial domain, then $\Dph$ and $\mathbb{S}_p(\Om)^H$ corresponds to the space of all radial functions in $\Dp$ and $\mathbb{S}_p(\Om)$ respectively. Next, we  consider a $H$-depended minimization problem analogous to \eqref{min}  $$\E_{g,\mu}^H(\Om) = \inf \left \{ J_{g,\mu}(u) : u \in \mathbb{S}_p(\Om)^H \right \}.$$
It is clear that $\E_{g,\mu}(\Om) \leq \E_{g,\mu}^H(\Om)$. It may happen that 
$\E_{g,\mu}^H(\Om)$ is attained in $\mathcal{D}_p(\Om)^H$ without $\E_{g,\mu}(\Om)$ being attained in $\mathcal{D}_p(\Om)$. 
Now one may ask, does a minimizer of $\E_{g,\mu}^H(\Om)$ actually solve $\eqref{Critical}$?
i.e., whether a critical point of $J_{g,\mu}$ over $\mathbb{S}_p(\Om)^H$ can be  a critical point of $J_{g,\mu}$ over $\mathbb{S}_p(\Om)$? 
The principle of symmetric criticality theory answers this question affirmatively. We used the following version of the principle of symmetric criticality by  Kobayashi and Ôtani. \cite[Theorem 2.2]{KOBAYASHI}.

\noi {\bf Theorem.} {\bf(Principle of symmetric criticality).}
{\it Let $V$ be a reflexive and strictly convex Banach space and  $H$ be a group that acts on $V$ isometrically i.e., $\|h(v)\|_V=\|v\|_V$, $\forall \, h \in H$, $v \in V$. If $F: V \mapsto \R$ is a $H$-invariant, ${\rm{C}}^1$, then 
$$ (F|_{V^H})'(v)=0 \ \mbox{implies} \ F'(v)=0  \ \mbox{and} \ v \in V^H \,,  $$
where $V^H$ is the set of all $H$-invariant elements of $V$.}

For the remaining part of this section, we make the following assumptions on $\Om$ and $g$:
\begin{enumerate}[label={($\bf H2$)}]
\item $\mbox{The domain} \ \Om \ \mbox{and the Hardy potential} \ g \ \mbox{are} \  H \mbox{-invariant} \,,$ where $H$ is a closed subgroup of $\mathcal{O}(N)$. \label{H2}  
\end{enumerate}
\noi For any closed subgroup $H$ of $\mathcal{O}(N)$, and $\Om,g$ as in \ref{H2}, it is easy to verify that
$$ \int_{\Om} [|\nabla u_h|^p - \mu g |u_h|^p ] \dx = \displaystyle \int_{\Om} [|\nabla u|^p - \mu g |u|^p ] \dx \,,$$
i.e., $J_{g,\mu}$ is $H$-invariant, and also $J_{g,\mu}$ is $\text{C}^1$ as $g \in \mathcal{H}_p(\Om).$ Thus,
the principle of symmetric criticality holds for $J_{g,\mu}$, and hence a minimizer of $\E_{g,\mu}^H(\Om)$ (if exists) solves $\eqref{Critical}$ (up to a constant multiple).  

Now, for $H,\Om,$ and $g$ as given in \ref{H1} and \ref{H2}, we investigate whether $\E_{g,\mu}^H(\Om)$ is achieved or not. For this purpose, analogous to $\C_{g,\mu}$,  we introduce a $H$-depended concentration function of $g$  as follows:
\begin{eqnarray*} \C_{g,\mu}^H(x):= \lim_{r \ra 0} \E_{g,\mu}^H(\Om \cap H(B_r(x))) \quad \ \mbox{and} \quad \
\C_{g,\mu}^{H}(\infty):=  \lim_{R \ra \infty} \E_{g,\mu}^H(\Om \cap B_R^c)  
  \end{eqnarray*}
\noi and $\C_{g,\mu}^{H,*}(\Om)  :=  \displaystyle \inf_{\overline{\Om}} \C_{g,\mu}^H(x) \,.$ Clearly, $\C_{g,\mu}^H$ is constant on each $H$-orbits, and if $H=\{Id_{\R^N}\}$, then $\C_{g,\mu}^H = \C_{g,\mu}$. Now, in the similar fashion as in Definition \ref{crisubcri}, we make the following definition:
\begin{definition} Let $g \in \mathcal{H}_p(\Om)$ be $H$-invariant. Then for a $\mu$  in $(0, \mu_1(g))$, we say 
\begin{enumerate}[(i)]
    \item $g$ is  {\bf{$H$-subcritical in $\Om$}} if $ \E_{g,\mu}^H(\Om)<\C_{g,\mu}^{H,*}(\Om)$,
    and {\bf{$H$-subcritical at infinity}} if $ \E_{g,\mu}^H(\Om)<\C_{g,\mu}^{H}(\infty)$.
    \item $g$ is {\bf{$H$-critical in $\Om$}} if $ \E_{g,\mu}^H(\Om)=\C_{g,\mu}^{H,*}(\Om)$,
    and {\bf{$H$-critical at infinity}} if $\C_{g,\mu}^{H}(\infty)= \E_{g,\mu}^H(\Om)$.
\end{enumerate}
\end{definition}
Now analogues to Theorem 1.3,  we have the following result:
\begin{theorem} \label{trivial}
Let $H$, $\Om$  and $g$ as given in {\rm{\ref{H1}} and \ref{H2}}. For  $\la \in (0,\mu_1(g))$, assume that $g$ is $H$-subcritical in $\Om$ and at infinity. Then \eqref{Critical} admits a positive solution for $\mu=\la$.  \end{theorem}

\begin{rmk}
Indeed, there are Hardy potentials that are not subcritical but $H$-subcritical for some $H$, see Example \ref{criticalatpts}-$(iii)$. 
\end{rmk}

Next, to understand the situation when  $g$ fails to be $H$-subcritical either in $\Om$ or  at infinity,  
we consider a minimizing seqeunce $(w_n)$ of $\E^H_{g,\mu}(\Om)$.  
In this case, $(w_n)$ can  concentrate only on a finite $H$-orbit in $\overline{\Om}$ and or at infinity (see Corollary \ref{Cor_ConCpct}). In particular, if all the $H$-orbits are infinite under the $H$-action, then $(w_n)$ can not concentrate anywhere in $\overline{\Om}$. 
These ideas leads to our next result.
\begin{theorem} \label{symmetricsol}
Let $H$, $\Om$ and $g$ be as given in {\rm{\ref{H1}} and \ref{H2}} such that the orbits $Hx$ are infinite for all $x \in \overline{\Om}$. Assume that for some 
  $\la \in (0,\mu_1(g))$, $g$ is $H$-subcritical at infinity. Then \eqref{Critical} admits a positive solution for $\mu=\la$.  \end{theorem}

\begin{rmk} \rm
For $2\le k\le N,$ let $H=\mathcal{O}(k)\times \{I_{N-k}\}$, where any $h \in H$ is considered as $\left(\begin{smallmatrix}
  h_k & 0_{N-k}\\
  0_{n-k} & I_{N-k}
\end{smallmatrix}\right)$; $h_k \in \mathcal{O}(k)$. Let
$\Om=\Om_k \times \Om_{N-k}$,  where $\Om_{N-k}$ is a general domain in $\R^{N-k}$ and
\begin{equation} \label{domain}
 \Om_k:= \left \{x \in \R^k :  a \leq |x| \leq b  \right \}^{\circ} \,, \ a,b \geq 0 \,, \end{equation}
$A^{\circ}$  denotes the interior of a set $A$.
Depending on the values of $a,b$, the above domain $\Om$ in \eqref{domain} can be a ball, an annuals, exterior of a ball, unbounded cylinders, or entire $\R^N$. In particular, if $k=N$, then $\Om$ is all radial domains in $\R^N.$ 
Observe that $\Om$ is $H$-invariant, and only the origin has a finite $H$-orbit. If $a>0$ then $0 \notin \overline{\Om},$ and hence Theorem \ref{symmetricsol} helps us to  identify $H$-invariant critical Hardy potentials on $\Om$ for which \eqref{Critical} admits a positive solution. For instance, for $k=N$ and $0<a<b<\infty$ (i.e., annular domains in $\R^N$),  $g \equiv 0$ is $O(N)$-subcritical at infinity (since the annular domains are bounded). Hence, for $g \equiv 0,$ \eqref{Critical} admits a positive solution. See Example \ref{criticalatpts} for more of such potentials.
\end{rmk} 
Next we address the case when $g$ is not necessarily $H$-subcritical at infinity. In this case, we assume that $\Om=\R^N$ and $g$ satisfies the following homogeneity condition:
\begin{enumerate}[label={($\bf H3$)}]
\item $g(rz) \geq \frac{g(z)}{r^p}; z \in \R^N, r >0 \,.$ \label{H3}  
\end{enumerate}
\begin{theorem} \label{atorigin}
Let $H$ be a closed subgroup of $\mathcal{O}(N)$ and $\Om=\R^N$ and $g$ be as in {\rm{\ref{H1}}} and {\rm{\ref{H2}}}.
\begin{enumerate}[(i)]
\item If $g$ is $H$-subcritical in $\R^N$ for some $\la \in (0,\mu_1(g))$ and satisfies {\rm{\ref{H3}}} for small values of $r>0$, then
\eqref{Critical} admits a positive solution for $\mu=\la$.
\item 
Let the orbits $Hx$ are infinite for all $x (\neq 0) \in \R^N.$ If $g$ satisfies {\rm{\ref{H3}}} for all $r>0$, then \eqref{Critical} admits a positive solution for all $\mu \in (0,\mu_1(g))$.
\end{enumerate}
\end{theorem}
\begin{rmk}
Consider $g(x)= 0 \ \text{or} \ \frac{1}{|x|^p}$ on $\R^N$ and $H=O(N)$. Then $g$ satisfies {\rm \ref{H3}}, and the above theorem assures that \eqref{Critical} admits a positive solution on entire $\R^N$ for all $\mu \in (0,\mu_1(g))$.
\end{rmk}

Next we provide a necessary condition for the existence of a positive solution for \eqref{Critical}.
In this direction, a Pohozave type identity is available for  bounded star-shaped domain \cite[Theorem 1.1]{Egneel}.  
Here we establish such an identity for entire $\R^N$.  
\begin{theorem} \label{ness}
Let $g \in {\rm{C}}^{\al}_{loc}(\R^N)$, for some $\al \in (0,1)$ and $u \in \mathcal{D}_p(\R^N)$ such that $g(x)|u|^p \in L^1(\R^N)$ and also $x. \nabla g(x)|u|^p \in L^1(\R^N).$ Further, if $u$ solves
$$-\De_p u - g |u|^{p-2}u = |u|^{p^*-2}u \ \ \mbox{in} \ \R^N \,,$$
in weak sense, then the following holds:
$$\int_{\R^N} [x.\nabla g(x) + p g(x)]|u|^p \dx=0 .$$
\end{theorem} 
\noi A similar identity has been derived when the problem does not involve critical exponent \cite[Proposition 4.5]{Tertikas}(for $p=2$) and \cite[Theorem 6.1.3]{Anoopthesis} (for general $p$). 

The rest of the paper is organised as follows. In Section \ref{Prelim}, we briefly discuss about the spaces $\H_p(\Om)$ and $\F_p(\Om)$, and recall some important results such as principle of symmetric criticality, a strong maximum principle. In Section \ref{CONCPCT}, we derive a $H$-depended concentration-compactness lemma. Section \ref{thmproof} is devoted to the discussions on subcritical potentials and it includes  our proofs for Theorem \ref{mainthm} and Theorem \ref{perfect}. In Section \ref{Cric_poten},  Theorem \ref{symmetricsol} and Theorem \ref{atorigin} are proved. We prove Theorem \ref{ness} in Section \ref{necessity}.

\section{Preliminaries} \label{Prelim}
In this section, we briefly discuss  many essential results that are required for the development of this article.

\subsection{Principle of symmetric criticality}
Let $(V,\|.\|_V)$ be a real Banach space and $BL(V,V)$ be the space of all bounded linear operator on $V$. Let $H$ be a group. A representation $\pi: H \mapsto BL(V,V)$ is a map that satisfies the following:
\begin{enumerate}[(i)]
    \item $\pi(e)v=v$; $\forall v \in V$, where $e$ is the identity element in $H$.
    \item $\pi(h_1h_2)v= \pi(h_1) \pi(h_2)v$, $\forall h_1,h_2 \in H$, $v \in V.$
\end{enumerate}
The action of $H$ on $V$ is said to be isometric if $\|\pi(h)v\|_V=\|v\|_V$, $\forall h \in H$, $v \in V$.
The subspace $V^H:=\{v \in V: \pi(h)v=v, \forall \,h \in H\}$ is called the $H$-invariant subspace of $V$. A functional $J:V\mapsto \R$ is called $H$-invariant if $J(\pi(h)v)=J(v), \forall \,h \in H, \forall \,v \in V $. A differentiable, $H$-invariant functional $J:V\mapsto \R$ is said to satisfy {\it the principle of symmetric criticality} if
$$ {\bf{(P)}}  \ (J|_{V^H})'(v)=0 \ \mbox{implies} \ J'(v)=0  \ \mbox{and} \ v \in V^H.  $$
In 1979, Palais \cite{Palais} has introduced the notion of principle of symmetric criticality. Since then many versions of this principle were proved e.g., \cite[Theorem 2.2]{KOBAYASHI}, \cite[Theorem 2.7]{KOBAYASHI}, \cite[Theorem 2.1]{article}. In this article, we use the following version due to Kobayashi and Ôtani \cite[Theorem 2.2]{KOBAYASHI}.

\begin{theorem}[Principle of symmetric criticality] \label{PSC}
\cite[Theorem 2.2]{KOBAYASHI} Let $V$ be reflexive and strictly convex and the action of the $H$ on $V$ is isometric. If $J$ is a $H$-invariant, ${\rm{C}}^1$ functional, then 
  {\bf{(P)}} holds. 
\end{theorem}
\begin{rmk} \rm
Let $\Om$ be a domain in $\R^N$. Consider $V=\Dp$ and $H$ be a closed subgroup of $\mathcal{O}(N)$ which acts on $V$ as $\pi(h)v = v_h$, where $v_h(x):=v(h^{-1}(x))$, $\forall x \in \Om$. One can verify the following:
\begin{enumerate} [(a)]
    \item $\Dp$ is reflexive and strictly convex,
    \item the action of $H$ on $V$ is isometric.
\end{enumerate}
Furthermore,  for $g\in \mathcal{H}_p(\Om)$, the functional $J_{g,\mu}$ is $\text{C}^1$.
Thus, by Theorem \ref{PSC}, {\bf{(P)}} holds for $J_{g,\mu}$. 
\end{rmk}

\subsection{The space of signed measures} \label{spacemeasure}
We denote the collection of all Borel sets in $\R^N$ by $\mathbb{B}(\R^N).$  Let $\mathbb{M} (\R^N)$ be the space of all regular, finite,  Borel signed-measures on $\R^N$. It is well known that $\mathbb{M} (\R^N)$ is a Banach space with respect to the norm $\norm{\nu}=|\nu|(\R^N)$ (total variation of $\nu$). By the Reisz Representation theorem \cite[Theorem 14.14, Chapter 14]{Border}, $\mathbb{M}(\R^N)$ is the dual of $\text{C}_0(\R^N) := \overline{\text{C}_c(\R^N)}$ in $L^{\infty}(\R^N)$.  A  sequence $(\nu_n)$ is said to be weak* convergent to $\nu$ in $\mathbb{M}(\R^N)$, if
 \begin{eqnarray*}
  \int_{\R^N} \phi \ d\nu_n \ra  \int_{\R^N} \phi \ d\nu, \ as \ n \ra \infty,  \forall\,  \phi \in \text{C}_0(\R^N)\,.
 \end{eqnarray*}
 In this case we write $\nu_n \wrastar \nu$. 
 
 The following result is a consequence of the Banach-Alaoglu theorem \cite[Chapter 5, Section 3]{Conway} which states that for any normed linear space $X$, the closed unit ball in  $X^*$ is weak* compact.
\begin{proposition} \label{BanachAlaoglu}
 Let $(\nu_n)$ be a bounded sequence in $\mathbb{M}(\R^N)$, then there exists $\nu \in \mathbb{M}(\R^N)$ such that $\nu_n \overset{\ast}{\rightharpoonup} \nu$ up to a subsequence.
\end{proposition}
The next proposition follows from the uniqueness part of the Riesz representation theorem \cite[Theorem 14.14, Chapter 14]{Border}.
\begin{proposition} \label{defmeasure}
 Let $\nu \in \mathbb{M}(\R^N)$ be a positive measure. Then for an open $V \subseteq \Om$,
 \[ \nu(V)= \sup \left \{ \int_{\R^N} \phi \ d\nu : 0 \leq \phi \leq 1, \phi \in \rm{C}_c^{\infty}(\R^N) \ with \ Supp(\phi) \subseteq V   \right \} \,,\]
and for any $E \in \mathbb{B}(\R^N)$, $\nu(E):= \inf \big\{ \nu(V) : E \subseteq V \ \mbox{and} \ V \text {is open } \big\}$.
\end{proposition}

The following result plays an important role in proving  the concentration compactness principle. For a proof we refer to  \cite[Lemma 1.2]{Lions2a}.
\begin{proposition}
  \label{representation}
 Let $\nu, \Gamma$ be two non-negative, bounded measures on $\R^N$ such that
 \begin{equation*}
 \left[\int_{\R^N} |\phi|^q \ d\nu \right]^{\frac{1}{q}} \leq C \left[ \int_{\R^N} |\phi|^p \ d\Gamma \right]^{\frac{1}{p}}, \ \forall \,\phi \in {\rm{C}}_c^{\infty}(\R^N),
 \end{equation*}
 for some $C>0$ and $1 \leq p < q <\infty.$ Then there exist a countable set $\left\{x_j \in \R^N : j \in \mathbb{J} \right\}$ and $ \nu_j \in (0, \infty)$ such that
 $$\nu = \displaystyle \sum_{j \in \mathbb{J}} \nu_j \delta_{x_j}.$$
\end{proposition}

\begin{definition}
A measure $\Upsilon \in \mathcal{M}(\R^N)$ is said to be concentrated on a Borel set $F$ if 
$$\Upsilon(E)=\Upsilon(E\cap F), \ \forall \,\, E \in \mathbb{B}(\R^N) \,.$$
\end{definition}
\noi If $\Upsilon $ is concentrated on $F$, then one can observe that  $\Upsilon  (F) = \norm{\Upsilon}.$

\noi For any measure $\Upsilon \in \mathbb{M}(\R^N)$ and a $E \in \mathbb{B}(\R^N)$, we denote the restriction of $\Upsilon$ on $E$ as $\Upsilon_E.$ Observe that $\Upsilon_E$ is concentrated on $E$.

\subsection{Brezis-Lieb lemma}
The next lemma is due to Brezis and Lieb (see Theorem 1 of \cite{Brezis-Lieb}). 
\begin{lemma}\label{Bresiz-Lieb}
 Let $(\Om, \A, \mu)$ be a measure space and $(f_n)$ be a sequence of complex
 -valued measurable functions which are uniformly bounded in $L^p(\Om, \mu)$ for some $0<p< \infty$. Moreover, if $(f_n)$ converges to $f$ a.e., then
\[\lim_{n \ra \infty} \left| \norm{f_n}_{(p, \mu)} - \norm{f_n-f}_{(p, \mu)} \right| = \norm{f}_{(p, \mu)}.\]
\end{lemma}
\noi We also  require the following inequality \cite[Lemma I.3]{Lions2a} that played an important role in the proof of Bresiz-Lieb lemma: for $a,b \in \R^N$,
\begin{equation} \label{inequality}
\left||a+b|^p-|a|^p \right|\leq \epsilon |a|^p + C(\epsilon,p)|b|^p 
\end{equation}
valid for each $\epsilon >0$ and  $0< p< \infty$.

\subsection{A strong maximum principle}
We use a version of strong maximum principle that holds for $p$-Laplacian for showing the positivity of the solution of \eqref{Critical}, see \cite[Proposition 3.2]{Kawohl}.
\begin{lemma} \label{strongmax}
 Let $u \in \Dp$ and $g \in L^1_{loc}(\Om)$ be two non-negative functions on $\Om$ such that $g u^{p-1} \in L^1_{loc}(\Om)$ and $u$ satisfies the following differential inequality (in the sense of distributions)
 $$-\De_p u + gu^{p-1} \geq 0 \ \ \mbox{in} \ \ \Om.$$
 Then either $u \equiv 0$ or $u>0$ in $\Om.$
\end{lemma}

\subsection{The spaces $\mathcal{H}_p(\Omega)$ and $\mathcal{F}_p(\Omega)$} \label{spaces}
We briefly recall the spaces $\mathcal{H}_p(\Om)$ and $\mathcal{F}_p(\Om)$ that have been introduced in \cite{New}. The space of Hardy potentials $\mathcal{H}_p(\Om)$ is defined as the set of all $g \in L^1_{loc}(\Om)$ such that the following Hardy type inequality holds:
$$\int_{\Omega} |g| |u|^p \dx \leq C  \int_{\Omega} |\nabla u|^p \dx, \forall\; u \in \Dp,$$
for some $C>0$. In \cite{New}, using Mazya's p-capacity, we provide a Banach function space structure on $\mathcal{H}_p(\Om)$. Recall that for $ F \cset \Omega,$ the $p$-capacity of $F$ relative to $\Om$ is defined as,
  \[{\text{Cap}_p(F,\Om)}  = \inf \left\{ \displaystyle \int_{\Omega} | \nabla u |^p \dx: u \in  \mathcal{N}_p (F) \right \},\]
  where $ \mathcal{N}_p (F)= \{ u \in \Dp : u \geq 1 \ \mbox{in a neighbourhood of}\; F \}$.
We define the Banach function space norm on $\mathcal{H}_p(\Om)$ as follows,
 \begin{eqnarray*}
  \norm{g}_{\mathcal{H}_p}= \sup\left\{ \frac{\int_{F} |g|\dx}{\text{Cap}_p(F,\Om)}:F \cset \Om; |F|\ne 0 \right\}.                           
   \end{eqnarray*} 
   
Now we define $\mathcal{F}_p(\Om) =  \overline{{\rm{C}}_c^{\infty}(\Om)} \text{ in } \ \mathcal{H}_p(\Om)$. The following result is proved in \cite[Theorem 1]{New}.
\begin{proposition} \label{pastthm}
 Let $g\in \mathcal{H}_p(\Om)$. Then  $G_p:\Dp \ra \R$ is compact if and only if $g \in \mathcal{F}_p(\Om)$. 
\end{proposition}

\section{A variant of concentration compactness principle} \label{CONCPCT} 

In this section, we derive a variant of the concentration compactness principle that we use extensively in this article.

{\bf{Notations.}} We will be following the notations below throughout this article.
\begin{itemize}
    \item For each $R>0$, we fix a function $\Phi_R \in \text{C}_b^{1}(\R^N)$ (bounded ${\rm{C}}^1$ functions on $\R^N$) satisfying  $0\leq \Phi_R \leq 1$, $\Phi_R = 0 $ on $\overline{B_R}$ and $\Phi_R = 1 $ on $B_{R+1}^c$.
    \item For a sequence $(u_n)$ in $\mathcal{D}_p(\R^N)^H$ with  $u_n  \wra u \ \mbox{in} \ \mathcal{D}_p(\R^N)$, each of the following sequences corresponds to a bounded  sequence in  $\mathbb{M}(\R^N)$ and hence  weak* converges to a measure in $\mathbb{M}(\R^N)$. For convenience, we denote the sequences and their limits as given below:  
 \begin{eqnarray*}
\nu_n:=    |u_n - u|^{p^*} &\wrastar& \nu \ \ \mbox{in} \ \ \mathbb{M}(\R^N)  \,, \\
 \Gamma_n:= |\nabla (u_n-u)|^p &\wrastar& \Gamma \ \ \mbox{in} \ \ \mathbb{M}(\R^N) \,, \\
  \gamma_n:= g |u_n-u|^p   &\wrastar& \gamma \ \ \mbox{in} \ \ \mathbb{M}(\R^N) \,, \\
  \widetilde{\Gamma}_n:=|\nabla u_n|^p    &\wrastar& \widetilde{\Gamma} \ \ \mbox{in} \ \ \mathbb{M}(\R^N) \,, \\
 \widetilde{\gamma}_n := g |u_n|^p   &\wrastar& \widetilde{\gamma} \ \ \mbox{in} \ \ \mathbb{M}(\R^N).
 \end{eqnarray*}
The following limits will be denoted by: 
 \begin{eqnarray*}
 \lim_{R \ra \infty} \overline{\lim_{n \ra \infty}} \int_{|x| \geq R} |u_n - u|^{p^*} \dx  &=& \nu_{\infty}  \\
 \lim_{R \ra \infty} \overline{\lim_{n \ra \infty}} \int_{|x| \geq R} |\nabla (u_n-u)|^p \dx &=&  \Gamma_{\infty}  \\
 \lim_{R \ra \infty} \displaystyle \overline{\lim_{n \ra \infty}} \int_{|x| \geq R} g |u_n-u|^p \dx  &=& \gamma_{\infty}  \,.
  \end{eqnarray*}

\end{itemize}

\begin{proposition} \label{equivalent_limit}
The following statements are true:
\begin{enumerate}[(i)]
    \item $ \displaystyle \lim_{R \ra \infty}\overline{\lim_{n\ra \infty}} \int_{\R^N} | u_n|^{p^*} \Phi_R \dx = \nu_{\infty} \,.$
    \item  $\displaystyle \lim_{R \ra \infty} \overline{\lim_{n\ra \infty}} \int_{\R^N} g| u_n|^{p} \Phi_R \dx = \gamma_{\infty} \,.$
    \item $\displaystyle \lim_{R \ra \infty} \overline{\lim_{n\ra \infty}} \int_{\R^N}  |\nabla u_n|^p  \Phi_R \dx= \Gamma_{\infty} \,. $
\end{enumerate}
\end{proposition}
\begin{proof} 
$(i)$
Use Brezis-Lieb lemma to obtain
 \begin{align} \label{nuinfinity2}
 \lim_{n \ra \infty} \left|\int_{|x|\geq R} |u_n-u|^{p^*} \dx - \int_{|x|\geq R} |u_n|^{p^*} \dx \right| = \int_{|x|\geq R} |u|^{p^*} \dx.
  \end{align}
Now, by the sub-additivity property of lisup, we get
\begin{align*} 
  & \,  \left|\overline{\lim_{n \ra \infty}}\int_{|x|\geq R} |u_n|^{p^*} \dx - \overline{\lim_{n \ra \infty}}\int_{|x|\geq R} |u_n-u|^{p^*} \dx \right| \\
  & \,   \quad \leq \overline{\lim_{n \ra \infty}} \left|\int_{|x|\geq R} |u_n-u|^{p^*} \dx - \int_{|x|\geq R} |u_n|^{p^*} \dx \right| \,.
 \end{align*}
Thus, using \eqref{nuinfinity2}, we obtain \begin{align} \label{equality55}
     \lim_{R\ra \infty} \overline{\lim_{n \ra \infty}}\int_{|x|\geq R} |u_n|^{p^*} \dx = \lim_{R\ra \infty} \overline{\lim_{n \ra \infty}}\int_{|x|\geq R} |u_n-u|^{p^*} \dx =\nu_{\infty} \,.
 \end{align}
 Notice that
  \begin{eqnarray*}
  \int_{|x|\geq R+1} |u_n |^{p^*} \dx \leq \int_{\R^N} |u_n|^{p^*} \Phi_R \dx \leq \int_{|x|\geq R} |u_n |^{p^*} \dx.
 \end{eqnarray*}
By taking $n, R \ra \infty$ and using \eqref{equality55} we get $(i)$.\\
$(ii), (iii)$ follows by the similar  set of calculations as given in $(i)$.
\end{proof}

\begin{proposition} \label{ConCpct:1}
Let $H$ be a closed subgroup of $\mathcal{O}(N)$, $g \in \mathcal{H}_p(\R^N)$ be a non-negative $H$-invariant Hardy potential, and $(u_n)$ be a sequence in $\mathcal{D}_p(\R^N)^H$ such that $u_n  \wra u \ \mbox{in} \ \mathcal{D}_p(\R^N)^H$. 
Then the followings hold:
\begin{enumerate}[(i)]
\item  if $\Phi \in {\rm{C}}_b^1(\Om)$ is such that $\nabla \Phi$ has compact support, then
  \begin{align} \label{takingout}
  \lim_{n \ra \infty} \int_{\Om} |\nabla((u_n -  u)\Phi)|^p \dx = \lim_{n \ra \infty} \int_{\Om} |\nabla(u_n -  u)|^p |\Phi|^p \dx.    
  \end{align}
\item there exists a countable set $\mathbb{J}$ such that $\nu =\sum_{j \in \mathbb{J}} \nu_j \delta_{x_j}$, where $\nu_j\in (0,\infty)$, $x_j \in \R^N$. In particular, $\nu$ is supported on the countable set $F_{\mathbb{J}}:=\{x_j \in \R^N: j \in \mathbb{J}\}$,
\item $\gamma$ is supported on $\overline{\sum_g}$ \,.
 \end{enumerate}
\end{proposition}
\begin{proof} $(i)$ 
  Let $\epsilon >0$ be given. Using \eqref{inequality},
   \begin{eqnarray*}
      \bigg| \int_{\Om} |\nabla((u_n  &-&  u)\Phi)|^p \dx - \int_{\Om} |\nabla(u_n -  u)|^p |\Phi|^p \dx \bigg|  \\
              &\leq &   \epsilon \int_{\Om} |\nabla(u_n -  u)|^p |\Phi|^p \dx + C(\epsilon,p) \int_{\Om} |u_n -  u|^p |\nabla \Phi|^p \dx. 
   \end{eqnarray*}
 Since $\nabla \Phi$ is compactly supported, by Rellich-Kondrachov compactness theorem \cite[Theorem 2.6.3]{KesavanPDE}, the second term in the right-hand side of the above inequality goes to 0 as $n \ra \infty$. Further, as $(u_n)$ is bounded in $\Dp$ and $\epsilon>0$ 
 is arbitrary, we obtain the desired result.

$(ii)$ Let $\mu \in (0,\mu_1(g))$. From the definition of $\E_{g,\mu}(\R^N)$, we have
\begin{align} \label{1ststep}
    \E_{g,\mu}(\R^N) \leq \displaystyle  \displaystyle\frac{\displaystyle \int_{\R^N} [|\nabla \phi (u_n-u)|^p-\mu g|\phi (u_n-u)|^p] \dx }{\left[\displaystyle\int_{\R^N} |\phi (u_n-u)|^{p^*} \dx \right]^{\frac{p}{p^*}}} \leq \displaystyle  \displaystyle\frac{\displaystyle \int_{\R^N} |\nabla \phi (u_n-u)|^p \dx}{\left[\displaystyle\int_{\R^N} |\phi (u_n-u)|^{p^*} \dx \right]^{\frac{p}{p^*}}}
\end{align}
for all $\phi \in \text{C}_c^{\infty}(\R^N).$ By the above assertion $(i)$ we have 
\begin{align*} 
    \lim_{n \ra \infty} \int_{\R^N} |\nabla \phi (u_n-u)|^p\dx = \lim_{n \ra \infty} \int_{\R^N} |\phi \nabla (u_n-u)|^p \dx \,.
\end{align*}
Thus, by taking $n \ra \infty$ in \eqref{1ststep},
$$\E_{g,\mu}(\R^N) \left[\int_{\R^N} |\phi|^{p^*} \ d\nu \right]^{\frac{p}{p^*}} \leq \int_{\R^N} |\phi|^p \ d \Gamma \,.$$
Now, $(ii)$ follows from Proposition \ref{representation}.

$(iii)$ For $\phi \in \text{C}_c^{\infty}(\R^N)$, $(u_n-u) \phi \in \Dp$, and since $g \in \H_p(\Om)$, it follows that
\begin{eqnarray*}
 \int_{\R^N} |\phi|^p \d \gamma_n  = \int_{\Om} g|(u_n-u)\phi|^p \dx & \leq &  \frac{1}{\mu_1(g,\Om)}  \int_{\Om} |\nabla((u_n -  u)\phi)|^p \dx  \\
& = & \frac{1}{\mu_1(g,\Om)}  \int_{\R^N} |\nabla((u_n -  u)\phi)|^p \dx.
\end{eqnarray*}
Take $n \ra \infty$ and use assertion $(i)$ to obtain
 \begin{align} \label{forrmk}
     \int_{\R^N} |\phi|^p \d \gamma \leq \frac{1}{\mu_1(g,\Om)} \int_{\R^N} |\phi|^p \d \Gamma .
 \end{align}
Now by the density of $\text{C}_c^{\infty}(\R^N)$ in $\text{C}_0(\R^N)$  together with Proposition \ref{defmeasure}, we get 
  \begin{equation}\label{measureinequality1}
   \gamma(E) \leq \frac{\Gamma (E)}{\mu_1(g,\Om)}   \ , \forall E \in \mathbb{B}(\R^N).
  \end{equation}
 In particular, $\gamma \ll \Gamma$ and hence by Radon-Nikodym theorem, 
 \begin{equation} \label{measureinequality}
  \gamma(E) = \int_E  \frac{\d \gamma}{\d \Gamma} \d\Gamma \ , \forall E \in \mathbb{B}(\R^N). 
 \end{equation}
 Further, by Lebesgue differentiation theorem (page 152-168 of \cite{Federer}) we have 
 \begin{equation} \label{Lebdiff}
  \frac{\d \gamma}{\d \Gamma}(x) = \lim_{r \ra 0} \frac{\gamma (B_r(x))}{\Gamma (B_r(x))}.
 \end{equation}
Now replacing $g$ by $g \chi_{B_r(x)}$ and proceeding as before, one can get an analogue of \eqref{measureinequality1}as follows:
 \[ \gamma(B_r(x)) \leq \frac{\Gamma (B_r(x))}{\mu_1(g,B_r(x))}.\] 
 Thus from \eqref{Lebdiff} we get 
\begin{eqnarray} \label{21}
 \frac{\d \gamma}{\d \Gamma} (x) \leq \frac{1}{\mu_1(g,x)} \,.
\end{eqnarray} 
Now from \eqref{measureinequality} and \eqref{21}, we conclude that $\gamma$ is supported on $\overline{\Sigma_g}$.
\end{proof}

As we have seen from the above proposition that $\nu$ is supported on the countable set $F_{\mathbb{J}}=\{x_j \in \R^N: j \in \mathbb{J}\}$. Now let us define  $$\Gamma_{F_{\mathbb{J}}}=\sum_{j \in \mathbb{J}} \Gamma_j \delta_{x_j} \,,  \gamma_{F_{\mathbb{J}}}=\sum_{j \in \mathbb{J}} \gamma_j \delta_{x_j} \,, \ \text{and} \ \zeta_\mu=\Gamma_{F_{\mathbb{J}}}-\mu \gamma_{F_{\mathbb{J}}}$$ for  $\mu \in (0,\mu_1(g))$. Then we have the following proposition.
\begin{proposition} \label{ConCpct:new}
Let $H$ be a closed subgroup of $\mathcal{O}(N)$, $g \in \mathcal{H}_p(\R^N)$ be a non-negative $H$-invariant Hardy potential, and $(u_n)$ be a sequence in $\mathcal{D}_p(\R^N)^H$ such that $u_n  \wra u \ \mbox{in} \ \mathcal{D}_p(\R^N)^H$. If $u=0$ and $\E_{g,\mu}^H(\R^N) \norm{\nu}^{\frac{p}{p^*}} = \norm{\zeta_\mu}$ for some $\mu \in (0,\mu_1(g))$, then $\nu$ is either zero or concentrated on a single finite
 $H$-orbit in $\R^N$ \,.
\end{proposition}

\begin{proof} Let $u=0$ and $\E_{g,\mu}^H(\R^N) \norm{\nu}^{\frac{p}{p^*}} = \norm{\zeta_\mu}$.  First we show that $\nu$ is supported on a single $H$ orbit.  By the definition of $\E_{g,\mu}^H(\R^N)$ 
\begin{align*}
\left[\int_{\R^N}  |\phi u_n|^{p*} \dx \right]^{\frac{p}{p*}} \leq [\E_{g,\mu}^H(\R^N)]^{-1} \int_{\R^N} [|\nabla (\phi u_n)|^p - \mu g |\phi u_n|^p] \dx
\end{align*}
for any $H$-invariant function
$\phi \in \mbox{C}_c^{\infty}(\R^N)$.
Since $u=0$, the above inequality together with \eqref{takingout} yield
\begin{equation*} 
\left[ \int_{\R^N}  |\phi|^{p*} \d\nu \right]^{\frac{p}{p*}}  \leq [\E_{g,\mu}^H(\R^N)]^{-1} \left[\int_{\R^N} |\phi|^p \d \Ga - \mu \int_{\R^N} |\phi|^p \d \gamma \right]
\end{equation*}
Consequently, $\nu^{\frac{p}{p*}} \leq [\E_{g,\mu}^H(\R^N)]^{-1} [\Gamma - \mu \gamma]$. Since $\nu$ is supported on $F_{\mathbb{J}}$, it follows that \begin{equation} \label{req1}
\nu^{\frac{p}{p*}} \leq [\E_{g,\mu}^H(\R^N)]^{-1} \zeta_\mu \,.
\end{equation}
Further, by applying Holder's inequality we get $\left[\int_{\R^N} |\phi|^p \d\zeta_\mu \right]^{\frac{p^*}{p}} \leq \left[\int_{\R^N} |\phi|^{p^*} \d\zeta_\mu \right] \|\zeta_\mu\|^{\frac{p^*}{p}-1}$. This gives $\zeta_\mu^{\frac{p^*}{p}} \leq \|\zeta_\mu\|^{\frac{p^*}{p}-1} \zeta_\mu$. Thus, \eqref{req1} gives
$\nu(E)  \leq [\E_{g,\mu}^H(\R^N)]^{-{{\frac{p^*}{p}}}} \|\zeta_\mu\|^{\frac{p^*}{N}} \zeta_\mu(E) $ for all $H$-invariant $E \in \mathbb{B}(\R^N)$.
Now, since $\E_{g,\mu}^H(\R^N) \norm{\nu}^{\frac{p}{p^*}} = \norm{\zeta_\mu}$, it follows that 
$$\nu (E) = [\E_{g,\mu}^H(\R^N)]^{-{\frac{p*}{p}}} \|\zeta_\mu\|^{\frac{p*}{N}} \zeta_\mu (E)$$ for all $H$-invariant $E \in \mathbb{B}(\R^N)$.
We use this equality in \eqref{req1}, and also $\E_{g,\mu}^H(\R^N) \norm{\nu}^{\frac{p}{p^*}} = \norm{\zeta_\mu}$ to obtain
$$\nu (E)^{\frac{1}{p*}} \nu (\R^N)^{\frac{1}{N}} \leq \nu (E)^{\frac{1}{p}} \,,$$
for any $H$-invariant $E \in \mathbb{B}(\R^N)$.
Thus, $\nu(E)$ is either $0$ or $\|\nu\|$, and hence $\nu$ is concentrated on a single $H$ orbit.
Now, let $\nu$ be concentrated on the orbit $H \xi$ for some $\xi \in \R^N$. We show that $H\xi$ is finite. It follows from the inequality \eqref{req1} and Lemma \ref{representation} that there exist a countable set $F_{\mathbb{J}}=\left\{x_j \in \R^N : j \in \mathbb{J} \right\}$ and $ \nu_j \in (0, \infty)$ such that
$\nu =  \sum_{j \in \mathbb{J}} \nu_j \delta_{x_j}.$
Since $\nu$ is concentrated at $H \xi$, it is clear that $ F_{\mathbb{J}} = H \xi.$ Noticing $\nu$ is invariant under any orthogonal transformations $h \in H$ (i.e., $\nu(h(E))=\nu(E)$ for all $E \in \mathbb{B}(\R^N)$ and $h\in H$), we infer that $\nu_i=\nu_j$ for all $i,j \in \mathbb{J}.$ Thus, $\mathbb{J}$ has to be finite (as $\|\nu\|<\infty$). 
\end{proof}

Now, we are ready to derive $(g,H)$-depended concentration compactness lemma.   For $g\equiv 0$, a similar result is obtained in \cite[Lemma 4.3]{Shoyeb}.   Here we obtain an analogous  result for the case $g\ne 0$ under the assumption  $|\overline{{\sum}_g}|=0$.
\begin{lemma} \label{ConCpct:2}
Let $H$ be a closed subgroup of $\mathcal{O}(N)$, $g \in \mathcal{H}_p(\R^N)$ be non-negative and $H$-invariant. Assume that $(u_n)$ is a sequence in $\mathcal{D}_p(\R^N)^H$ such that $u_n  \wra u \ \mbox{in} \ \mathcal{D}_p(\R^N)^H$. If $|\overline{{\sum}_g}|=0$,  
then for $\mu \in (0,\mu_1(g))$, the following holds:
\begin{enumerate}[(a)]
 \item $  \C_{g,\mu}^{H,*}(\R^N) \norm{\nu}^{\frac{p}{p^*}} + \mu \norm{\gamma} \leq  \norm{\Gamma_{\overline{\sum_g} \cup F_{\mathbb{J}}}} ,$  
 \item $ \C_{g,\mu}^{H}(\infty) \nu_{\infty}^{\frac{p}{p^*}} +  \mu \gamma_{\infty}\leq \Gamma_{\infty},$
\item $\displaystyle \overline{\lim_{n \ra \infty}} \displaystyle \int_{\R^N} |u_n|^{p^*} \dx = \int_{\R^N} |u|^{p^*} \dx + \norm{\nu} + \nu_{\infty},$
\item  $\displaystyle \overline{\lim_{n \ra \infty}} \displaystyle \int_{\R^N} g|u_n|^{p} \dx = \int_{\R^N} g|u|^{p} \dx + \norm{\gamma} + \gamma_{\infty},$
 \item 
          $\displaystyle\overline{\lim_{n \ra \infty}} \displaystyle\int_{\R^N} |\nabla u_n|^p \dx   \geq  \displaystyle \int_{\R^N} |\nabla u|^p \dx  + \norm{\Gamma_{\overline{\sum_g} \cup F_{\mathbb{J}}}} + \Gamma_{\infty} $,
 \item 
$\displaystyle \overline{\lim_{n \ra \infty}} \displaystyle\int_{\R^N} \left[|\nabla u_n|^p - \mu g |u_n|^p \right]\dx \geq\displaystyle \int_{\R^N} \left[|\nabla u|^p - \mu g |u|^p \right] \dx + \C_{g,\mu}^{H,*}(\R^N) \norm{\nu}^{\frac{p}{p^*}} + \C_{g,\mu}^{H}(\infty) \nu_{\infty}^{\frac{p}{p^*}}.$

 \end{enumerate}
\end{lemma}
\begin{proof} 

$(a)$ By the definition of $\E_{g,\mu}^H(H(B_r(x)))$, we have 
$$\E_{g,\mu}^H(H(B_r(x))) \leq \displaystyle  \frac{\displaystyle \int_{\R^N} [|\nabla \phi (u_n-u)|^p-\mu g|\phi (u_n-u)|^p] \dx }{\left[\displaystyle\int_{\R^N} |\phi (u_n-u)|^{p^*} \dx \right]^{\frac{p}{p^*}}}$$
for all $H$-inavriants $\phi \in \text{C}_c^{\infty}( H(B_r(x))).$ By taking $n \ra \infty$ in the above inequality and using \eqref{takingout}, we get
$$\E_{g,\mu}^H( H(B_r(x))) \left[\int_{\R^N} |\phi|^{p^*} \d\nu \right]^{\frac{p}{p^*}} + \mu \int_{\R^N} |\phi|^p \d\gamma \leq \int_{\R^N} |\phi|^p \d \Gamma \,.$$
In particular, for $x_j$ with $j\in \mathbb{J}$, by taking $\phi =1$ on $H(\{x_j\})$ and then letting $r \ra 0$, we get
$$\C_{g,\mu}^H(x_j) \ |\nu (Hx_j)|^{\frac{p}{p^*}} + \mu \gamma (Hx_j) \leq \Gamma (Hx_j) \,\forall j\in \mathbb{J}.$$
Taking the sum over $\mathbb{J}$ and using the  concavity of the map $f(t):=t^{\frac{p}{p^*}}$, we obtain
\begin{align} \label{altomic}
    \C_{g,\mu}^{H,*}(\R^N) \ \|\nu \|^{\frac{p}{p^*}} + \mu \|\gamma_{F_{\mathbb{J}}}\|  \leq \|\Gamma_{F_{\mathbb{J}}}\| \,.
\end{align}
Next we show that $\mu \|\gamma_{{\overline{\sum_g} \setminus F_{\mathbb{J}}}}\| \leq \|\Gamma_{{\overline{\sum_g} \setminus F_{\mathbb{J}}}}\|$. Since $g \in \H_p(\Om)$ and $\mu \in (0,\mu_1(g))$, we have
$$\mu \int_{\R^N} g |(u_n-u)\phi|^p \dx \leq \frac{\mu}{\mu_1(g)} \int_{\R^N} |\nabla (u_n-u) \phi|^p \dx \leq  \int_{\R^N} |\nabla (u_n-u) \phi|^p \dx $$
for any $\phi \in {\rm{C}}_c^{\infty}(\R^N)$. By taking $n\ra \infty$ and using \eqref{takingout}, one can get
$$\mu \int_{\R^N}  |\phi|^p \d\gamma \leq  \int_{\R^N} |\nabla  \phi|^p \d\Gamma $$
for any $\phi \in C_c^{\infty}(\R^N)$. Thus, we have $\mu \gamma({{\overline{\sum_g} \setminus F_{\mathbb{J}}}}) \leq \Gamma ({{\overline{\sum_g} \setminus F_{\mathbb{J}}}})$. In particular, we get 
\begin{align} \label{for_sum}
  \mu \|\gamma_{{\overline{\sum_g} \setminus F_{\mathbb{J}}}}\| \leq \|\Gamma_{{\overline{\sum_g} \setminus F_{\mathbb{J}}}}\| \,. 
\end{align} 
Now, by adding \eqref{for_sum} and \eqref{altomic}, and using the fact that $\nu$, $\Gamma$ are supported on $F_{\mathbb{J}}$, $\overline{\Sigma_g}$, we get
\begin{align} \label{atomic}
    \C_{g,\mu}^{H,*}(\R^N) \ \|\nu \|^{\frac{p}{p^*}} + \mu \|\gamma\|  \leq \|\Gamma_{{\overline{\sum_g} \cup F_{\mathbb{J}}}}\| \,.
\end{align}

$(b)$ For $R>0$, choose $\Phi_R \in \text{C}_b^{1}(\R^N)^H$ satisfying  $0\leq \Phi_R \leq 1$, $\Phi_R = 0 $ on $\overline{B_R}$ and $\Phi_R = 1 $ on $B_{R+1}^c$. Then $(u_n-u)\Phi_R \in \D_p(B_R^c)^H$. In order to prove $(b)$, one can start with $\E_{g,\mu}^H( B_R^c)$ instead of $\E_{g,\mu}^H(H(B_r(x)))$ and follow the similar arguments as in $(a)$ to get 
\begin{align} \label{ineq:1}
    \E_{g,\mu}^H( H(B_R^c)) \lim_{n \ra \infty} \left[\int_{\R^N} |\Phi_R|^{p^*}|u_n-u|^{p^*} \dx \right]^{\frac{p}{p^*}} & + \mu \lim_{n \ra \infty} \int_{\R^N} g|u_n-u||\Phi_R|^p \dx \nonumber \\
    & \leq \lim_{n \ra \infty} \int_{\R^N} |\Phi_R|^p |\nabla (u_n-u)|^p \dx  \,.
\end{align}
Thus, by taking $R \ra \infty$ in \eqref{ineq:1} and using Proposition \ref{equivalent_limit}, we prove the assertion $(b).$

$(c)$ Using Brezis-Lieb lemma together with Proposition \ref{equivalent_limit}-$(i)$, we have
\begin{align*}
& \, \overline{\lim_{n\ra \infty}}  \int_{\R^N} | u_n|^{p^*} \dx \\ & = \lim_{R\ra \infty}  \overline{\lim_{n\ra \infty}} \left[ \int_{\R^N} | u_n|^{p^*} (1-\Phi_R) \dx + \int_{\R^N} | u_n|^{p^*} \Phi_R \dx  \right] \\
 &= \lim_{R\ra \infty} \overline{\lim_{n\ra \infty}} \left[  \int_{\R^N} | u|^{p^*} (1-\Phi_R) \dx  + \int_{\R^N} | u_n-u|^{p^*} (1-\Phi_R) \dx + \int_{\R^N} | u_n|^{p^*} \Phi_R \dx \right]   \\
                                          &= \int_{\R^N} | u|^{p^*} \dx + \norm{\nu} + \nu_{\infty}. 
\end{align*}

$(d)$ As in $(c)$, using Brezis-Lieb lemma together with Proposition \ref{equivalent_limit}-$(ii)$, we deduce
\begin{align*}
& \,  \overline{\lim_{n\ra \infty}} \int_{\R^N} g| u_n|^{p} \dx \\
& =  \lim_{R\ra \infty} \overline{\lim_{n\ra \infty}} \left[ \int_{\R^N} g| u_n|^{p} (1-\Phi_R) \dx + \int_{\R^N} g| u_n|^{p} \Phi_R \dx  \right] \\
                                          &= \lim_{R\ra \infty} \overline{\lim_{n\ra \infty}} \left[  \int_{\R^N} g| u|^{p} (1-\Phi_R) \dx  + \int_{\R^N} g| u_n-u|^{p} (1-\Phi_R) \dx + \int_{\R^N} g| u_n|^{p} \Phi_R \dx \right]   \\
                                          &= \int_{\R^N} g| u|^{p} \dx + \norm{\gamma} + \gamma_{\infty}. 
\end{align*}

$(e)$ We break this proof into several steps.


\noi {\bf{Claim 1:}} $\widetilde{\Gamma}_{F_{\mathbb{J}}}=\Gamma_{F_{\mathbb{J}}}$. 
Let $\phi_{\var} \in \text{C}_c^{\infty}(B_{\var}(\om))$ satisfying  $0\leq \phi_{\var} \leq 1$, $\phi_{\var}(\om) = 1$, where $\om \in F_{\mathbb{J}}$. Then, using \eqref{inequality}, we have the following:
\begin{align*}
 \left|\widetilde{\Gamma}(\phi_{\var})-\Gamma(\phi_{\var}) \right| & = \overline{\lim_{n \ra \infty}} \int_{\R^N} \big|[\nabla (u_n)|^p-|\nabla u_n-u|^p]   \big| \phi_{\var} \dx \\
&   \leq \varepsilon \int_{\R^N} \phi_{\var} |\nabla u_n|^p  \dx + C(\varepsilon,p)\int_{\R^N} \phi_{\var} |\nabla u|^p  \dx \,.  
\end{align*}
 By taking $\var \ra 0$, we get $\widetilde{\Gamma}(\om)=\Gamma_{F_{\mathbb{J}}}(\om)$.

\noi {\bf{Claim 2:}} $\tilde{\Gamma}=\Gamma$, on $\overline{\sum_g}.$ Let $E\subset\overline{\sum_g}$ be a Borel set. Thus, for each $m \in \N$, there exists an open subset $O_{m}$ containing $E$ such that $|O_m|=|O_{m} \setminus E| < \frac{1}{m}$. Let $\var >0$ be given. Then, for any $\phi \in C_c^{\infty}(O_{m})$ with $0 \leq \phi \leq 1$, using \eqref{inequality} we have
\begin{align*}
    \left|\int_{\Om}  \phi \d\Gamma_n  -\int_{\Om}  \phi \d\tilde{\Gamma}_n  \right|&= \left|\int_{\Om}  \phi |\nabla (u_n-u)|^p \dx -\int_{\Om}  \phi |\nabla u_n|^p \dx \right| \\
    &\leq \var \int_{\Om}  \phi |\nabla u_n|^p \dx + \text{C}(\var,p) \int_{\Om} \phi |\nabla u|^p \dx \\
 & \leq   \var L + \text{C}(\var,p) \int_{O_{m}}   |\nabla u|^p \dx,
\end{align*}
where $L=\sup_{n}\left\{\int_{\Om} |\nabla u_n|^p \dx\right\}$. 
Now letting $n \ra \infty$, we obtain $
 \left|\int_{\Om}  \phi \d\Gamma  -\int_{\Om}  \phi \d\tilde{\Gamma}  \right|  \leq \var L + \text{C}(\var,p) \int_{O_{m}}   |\nabla u|^p \dx.
$
Therefore,
\begin{eqnarray*} 
\left|\Gamma (O_m)-\tilde{\Gamma} (O_m) \right| &=& \sup \left\{ \left|\int_{\Om}  \phi \d\Gamma -\int_{\Om}  \phi \d\tilde{\Gamma} \right|: \phi \in C_c^{\infty}(O_m), 0 \leq \phi \leq 1 \right \} \\
&\leq & \var L + \text{C}(\var,p) \int_{O_{m}}   |\nabla u|^p \dx,
\end{eqnarray*}
Now as $m \ra \infty,$ $|O_m|\ra 0$ and hence $| \Gamma (E)-\tilde{\Gamma} (E)| \leq \var L.$ Since $\var >0$ is arbitrary, we conclude $\Gamma(E)=\tilde{\Gamma} (E).$

\noi {\bf{Claim 3}:} $\|\widetilde{\Gamma}\| \geq  \displaystyle \int_{\R^N} |\nabla u|^p \dx + \|\Gamma_{\overline{\sum_g} \cup F_{\mathbb{J}}}\|$. Choose an arbitrary $\phi \in \text{C}_c^{\infty}(\R^N)$ with $0\leq \phi \leq 1$.
Hence, by the lower semicontinuity, we have
$$\displaystyle \lim_{n \ra \infty} \int_{\R^N} |\nabla u_n|^p \phi \ \dx \geq \displaystyle \int_{\R^N} |\nabla u|^p \phi \ \dx \,.$$
This yields $\widetilde{\Gamma} \geq |\nabla u|^p \dx$.  On the other hand,  the measure $|\nabla u|^p \dx$ is  singular to $\Gamma_{\overline{\sum_g} \cup F_{\mathbb{J}}}$ (as $\overline{\sum_g} \cup F_{\mathbb{J}}$ has Lebesgue measure zero)  and hence \begin{align} \label{step3}
    \|\widetilde{\Gamma}\| \geq \displaystyle \int_{\R^N} |\nabla u|^p \dx + \|\Gamma_{\overline{\sum_g} \cup F_{\mathbb{J}}}\|.
\end{align}

Now we are ready to prove $(e)$. Observe that
\begin{eqnarray*}
 \overline{\lim_{n\ra \infty}} \int_{\R^N} | \nabla u_n|^{p} \dx =  \overline{\lim_{n\ra \infty}}  \int_{\R^N} | \nabla u_n|^{p} (1-\Phi_R) \dx + \overline{\lim_{n\ra \infty}}  \int_{\R^N}  | \nabla u_n|^{p} \Phi_R  \dx \,.
\end{eqnarray*}
Using Proposition \ref{equivalent_limit}-$(iii)$, \eqref{step3}, and \eqref{altomic}, we infer that
\begin{align*}
    \overline{\lim_{n\ra \infty}} \int_{\R^N} | \nabla u_n|^{p} \dx =\|\widetilde{\Gamma}\|+\Gamma_{\infty}  \geq \displaystyle \int_{\R^N} |\nabla u|^p \dx + \|\Gamma_{\overline{\sum_g} \cup F_{\mathbb{J}}}\|+\Gamma_{\infty} \,.
\end{align*}

$(f)$ Since $\mu \in (0,\mu_1(g))$, it follows that $\int_{\R^N} \left[|\nabla u_n|^p - \mu g |u_n|^p \right] \dx \geq 0.$ Thus,  using the sub-additivity of limsup, $(e), (d)$ we get
$$\overline{\lim_{n \ra \infty}} \displaystyle\int_{\R^N} \left[|\nabla u_n|^p - \mu g |u_n|^p \right] \dx \geq \displaystyle \int_{\R^N} \left[|\nabla u|^p - \mu g |u|^p \right] \dx + (\|\Gamma_{\overline{\sum_g} \cup F_{\mathbb{J}}}\| -\mu \|\gamma\|)+(\Gamma_{\infty} -\mu \gamma_{\infty}) \,.$$
Now use $(a)$ and $(b)$ to obtain $(f)$.

\end{proof}

The next corollary state the above results for a general domain in place of $\R^N$. 

\begin{corollary} \label{Cor_ConCpct}
Let $H$ be a closed subgroup of $\mathcal{O}(N)$, and $g \geq 0, \Om $ be as in {\rm{(H2)}}. Assume that $(u_n)$ be a sequence in $\mathcal{D}_p(\Om)^H$ such that $u_n  \wra u \ \mbox{in} \ \mathcal{D}_p(\Om)^H$. If $|\overline{{\sum}_g}|=0$,  
then for $\mu \in (0,\mu_1(g))$, the following holds:
\begin{enumerate}[(a)]
\item there exists a countable set $\mathbb{J}$ such that $\nu =\sum_{j \in \mathbb{J}} \nu_j \delta_{x_j}$, where $\nu_j\in (0,\infty)$, $x_j \in \overline{\Om}$. In particular, $\nu$ is supported on the countable set $F_{\mathbb{J}}:=\{x_j \in \overline{\Om}: j \in \mathbb{J}\}$,
\item $\gamma$ is supported on $\overline{\sum_g},$
\item $  \C_{g,\mu}^{H,*}(\Om) \norm{\nu}^{\frac{p}{p^*}} + \mu \norm{\gamma} \leq  \norm{\Gamma_{\overline{\sum_g} \cup F_{\mathbb{J}}}} $,  
 \item $ \C_{g,\mu}^{H}(\infty) \nu_{\infty}^{\frac{p}{p^*}} +  \mu \gamma_{\infty}\leq \Gamma_{\infty},$
\item $\displaystyle \overline{\lim_{n \ra \infty}} \displaystyle \int_{\Om} |u_n|^{p^*} \dx = \int_{\Om} |u|^{p^*} \dx + \norm{\nu} + \nu_{\infty},$
\item $\displaystyle \overline{\lim_{n \ra \infty}} \displaystyle \int_{\Om} g|u_n|^{p} \dx = \int_{\Om} g|u|^{p} \dx + \norm{\gamma} + \gamma_{\infty},$
\item 
          $\displaystyle \overline{\lim_{n \ra \infty}} \displaystyle\int_{\Om} |\nabla u_n|^p \dx    \geq  \displaystyle \int_{\Om} |\nabla u|^p \dx  + \norm{\Gamma_{\overline{\sum_g} \cup F_{\mathbb{J}}}} + \Gamma_{\infty} $,
 \item 
$\displaystyle \overline{\lim_{n \ra \infty}} \displaystyle\int_{\Om} \left[|\nabla u_n|^p - \mu g |u_n|^p \right] \dx \geq\displaystyle \int_{\Om} \left[|\nabla u|^p - \mu g |u|^p \right] \dx + \C_{g,\mu}^{H,*}(\Om) \norm{\nu}^{\frac{p}{p^*}} + \C_{g,\mu}^{H}(\infty) \nu_{\infty}^{\frac{p}{p^*}}$
\item if $u=0$ and $\E_{g,\mu}^H(\Om) \norm{\nu}^{\frac{p}{p^*}} = \norm{\zeta_\mu}$, where $\zeta_\mu=\sum_{j \in \mathbb{J}} (\Gamma_j-\mu \gamma_j) \delta_{x_j}$, then  $\nu$ is either zero or concentrated on a single finite
 $H$-orbit in $\R^N$ \,
 \end{enumerate}
\end{corollary}
\begin{proof}
Recall that any function $u \in \mathcal{D}_p(\Om)$ can be considered as a $\mathcal{D}_p(\R^N)$ function whenever convenient. Since $g, (u_n), u$ are supported inside $\Om$, it is not difficult to see that $(i)$ $\nu, \Gamma =0$ outside $\overline{\Om}$, $(ii)$ $\C_{g,\mu}^{H,*}(\Om)=\C_{g,\mu}^{H,*}(\R^N)$. Hence the corollary follows as a consequence of Proposition \ref{ConCpct:1}, Proposition \ref{ConCpct:new}, and Lemma \ref{ConCpct:2}. 
\end{proof}

\begin{rmk} \label{bhul1} \rm



$(i)$  It follows from Corollary \ref{Cor_ConCpct}-$(c)$ that $  \E_{g,\mu}^{H}(\Om) \norm{\nu}^{\frac{p}{p^*}} + \mu \norm{\gamma} \leq  \norm{\Gamma_{\overline{\sum_g} \cup F_{\mathbb{J}}}} $. Now, if $  \E_{g,\mu}^{H}(\Om) \norm{\nu}^{\frac{p}{p^*}} < \norm{\zeta_\mu} $ ($\zeta_\mu$ is as in Proposition \ref{ConCpct:new}), then the previous inequality has to be strict. Therefore, if $  \E_{g,\mu}^{H}(\Om) \norm{\nu}^{\frac{p}{p^*}} + \mu \norm{\gamma} = \norm{\Gamma_{\overline{\sum_g} \cup F_{\mathbb{J}}}} $, then $  \E_{g,\mu}^{H}(\Om) \norm{\nu}^{\frac{p}{p^*}} = \norm{\zeta_\mu} $.

$(ii)$ Let $H$ be an infinite closed subgroup on $\mathcal{O}(N)$ and $0 \notin \overline{\Om}$. Assume all the hypothesis of Proposition \ref{ConCpct:new} (or Corollary \ref{Cor_ConCpct}-$(i)$). Then, it follows that $\nu=0.$ 

$(iii)$ Observe that, for a bounded sequence $(u_n)$ in $\Dp$, the measure $\nu$ in the above corollary helps us to determine whether $u_n \ra u$ in $L^{p^*}(\Om)$ or not.  
Precisely, if $\nu=0=\nu_{\infty}$, then $u_n \ra u$ in $L^{p^*}(\Om)$.
\end{rmk}
  

\section{Subcritical Potentials} \label{thmproof}
This section is devoted to prove Theorem \ref{mainthm} and Theorem \ref{perfect}. We bring up several important consequences of these theorems, and provide a few examples of subcritical and non-subcritical potentials.
Let us commence with the following proposition.
\begin{proposition} \label{req_proposition}
 Let $g \in \mathcal{H}_p(\Om)$ be such that $ g^{-} \in \mathcal{F}_p(\Om).$ Then $\C_{g,\mu}^*(\Om)=\C_{g^+,\mu}^*(\Om)$ and $\C_{g,\mu}(\infty)=\C_{g^+,\mu}(\infty)$ for all $\mu \in (0,\mu_1(g)).$
\end{proposition}
\begin{proof}
Let us fix $x \in \overline{\Om}$ and $u_n \in \mathbb{S}_p(\Om \cap B_{\frac{1}{n}}(x))$ be such that 
$$\C_{g^+,\mu}(x)=\displaystyle \lim_{n \ra \infty} \int_{\Om} [|\nabla u_n|^p - \mu g^+|u_n|^p] \dx   \,.$$
Since $\mu \in (0,\mu_1(g))$, the quasi-norm
$\|u\|_{\D_p,\mu}:= \left(\displaystyle\int_{\Om} [|\nabla u|^p - \mu g^+ |u|^p] \dx \right)^{\frac{1}{p}}$ is equivalent to the norm $\|.\|_{\D_p}$ on $\Dp.$ Thus, $(u_n)$ is bounded in $\Dp$. Furthermore, the supports of $(u_n)$ converges to the singleton set $\{x\}$.  
Consequently, 
$u_n \wra 0$ in $\mathcal{D}_p(\Om)$ (up to a subsequence). Now, since $g^- \in \mathcal{F}_p(\Om)$, we have
$\displaystyle \lim_{n \ra \infty} \int_{\Om} g^-|u_n|^p \dx=0$ (by Proposition \ref{pastthm}). Hence,
\begin{align*}
    \C_{g^+,\mu}(x) &=\displaystyle \lim_{n \ra \infty} \int_{\Om} [|\nabla u_n|^p - \mu g^+|u_n|^p] \dx + \lim_{n \ra \infty} \int_{\Om} \mu g^{-} |u_n|^p \dx \\
    & =\displaystyle \lim_{n \ra \infty} \int_{\Om} [|\nabla u_n|^p - \mu g|u_n|^p] \dx \\ & \geq \C_{g,\mu}(x) \,.
\end{align*}
The other way inequality holds trivially. Hence
$\C_{g,\mu}(x) = \C_{g^+,\mu} (x)$. Since $x \in \overline{\Om}$ is arbitrary, we prove that $\C_{g,\mu}^*(\Om) = \C_{g^+,\mu}^* (\Om)$. The other assertion follows from the similar set of arguments.
\end{proof}

\noi {\bf{Proof of Theorem \ref{mainthm} :}} We choose $H=\{Id_{\R^N}\}$ in Corollary \ref{Cor_ConCpct}. Then 
$\mathcal{D}_p(\Om)^H=\mathcal{D}_p(\Om)$, and $\mathbb{S}_p(\Om)^H=\mathbb{S}_p(\Om).$
Let $u_n \in \mathbb{S}_p(\Om)$ be a minimizing sequence of  $\E_{g,\mu}(\Om).$ Since $\mu \in (0,\mu_1(g))$, the quasi-norm
$\|u\|_{\D_p,\mu}:= \left(\displaystyle\int_{\Om} [|\nabla u|^p - \mu g |u|^p] \dx \right)^{\frac{1}{p}}$ is equivalent to the norm $\|.\|_{\D_p}$ on $\Dp.$  
Consequently, $(u_n)$ is bounded in $\Dp$, which implies 
$u_n \wra u$ in $\mathcal{D}_p(\Om)$ (up to a subsequence). Now, we use Proposition \ref{req_proposition}, and Corollary \ref{Cor_ConCpct} to obtain the following:
\begin{eqnarray*}
 \E_{g,\mu}(\Om) &=& \lim_{n \ra \infty} \int_{\Om} [|\nabla u_n|^p-\mu g |u_n|^p ] \dx   \\
      &\geq & \int_{\Om} \left[|\nabla u|^p-\mu g |u|^p \right] \dx +\C_{g,\mu}^{*}(\Om) \|\nu \|^{\frac{p}{p^{*}}} 
      +\C_{g,\mu}(\infty) \nu_{\infty}^{\frac{p}{p^{*}}} \\
      &\geq& \E_{g,\mu}(\Om) \left[\int_{\Om} |u|^{p^{*}} \dx \right]^{\frac{p}{p^{*}}} + \C_{g,\mu}^{*}(\Om) \|\nu \|^{\frac{p}{p^{*}}} 
      +\C_{g,\mu}(\infty) \nu_{\infty}^{\frac{p}{p^{*}}} \,.
\end{eqnarray*}
If one of $\|\nu \|$ or $\nu_{\infty}$ is non-zero, then using the hypothesis that $\E_{g,\mu}(\Om)<\C_{g,\mu}^{*}(\Om)$ and  $\E_{g,\mu}(\Om)<\C_{g,\mu}(\infty)$ (i.e., $g$ is subcritical in $\Om$ and at infinity), we infer 
$$\E_{g,\mu}(\Om) >\E_{g,\mu}(\Om) \left(\int_{\Om} |u|^{p^{*}} \dx + \norm{\nu} + \nu_{\infty}\right)^{\frac{p}{p^{*}}}.$$
By $(e)$ of Corollary \ref{Cor_ConCpct}, $\displaystyle \int_{\Om} |u|^{p^{*}} \dx + \norm{\nu} + \nu_{\infty}=\lim_{n \ra \infty} \|u_n\|_{p^*}^{p^*}=1$ (as $u_n \in \mathbb{S}_p(\Om)$). Hence, 
$\E_{g,\mu}(\Om)>\E_{g,\mu}(\Om)$, a contradiction. Therefore, $\|\nu \| =0=\nu_{\infty}$. As a consequence $\|u\|_{p^{*}}=1$.
Hence, $\E_{g,\mu}(\Om)$ is attained at $u$. Thus, $v=[\E_{g,\mu}(\Om)]^{\frac{1}{p^{*}-p}}u$ is a non-trivial solution to \eqref{Critical}. Notice that, if $u$ is a minimizer of $\E_{g,\mu}(\Om)$, then $|u|$ is also so.
Thus, we prove that there exists a non-negative, non-trivial solution $v$ of \eqref{Critical}. Therefore, we have
$$-\De_p v + \mu g^-v^{p-1} = \mu g^+v^{p-1} \geq 0,$$
in distribution sense, and Lemma \ref{strongmax} ensures that $v$ is positive.

\begin{rmk} \label{miniexist} \rm
The above proof gives not only the existence of a positive solution of \eqref{Critical} but also it assures that $\E_{g,\mu}(\Om)$ is attained at some positive $u \in \Dp.$ \end{rmk}

In the next proposition, we prove a particular case of Theorem \ref{perfect}.
\begin{proposition} \label{constants}
 \rm 
Let $N \geq p^2$ and $\Om$ be a bounded domain in $\R^N.$ Then any positive constant function $g$ is  sub-critical in $\Om$ as well as at infinity for all $\mu \in (0,\mu_1(g))$.
\end{proposition}
\begin{proof}
Without loss of generality, we consider the constant function ${\bf{1}}(z)= 1$; $z \in \Om$ and prove that ${\bf{1}}$ is sub-critical in $\Om$ and at infinity. Notice that $\mu_1({\bf{1}})=\la_1$, where $\la_1$ is the first Dirichlet eigenvalue of $\De_p$. Fix $\mu \in (0,\la_1)$.  For $x \in \overline{\Om}$, there exists $v_{n} \in \mathbb{S}_p(\Om \cap B_{\frac{1}{n}}(x))$ (by definition of $\E_{{\bf{1}},\mu}(\Om \cap B_{\frac{1}{n}}(x))$) such that
\begin{align} \label{for_constant}
    \int_{\Om} [|\nabla v_n|^p - \mu |v_n|^p] \dx < \E_{{\bf{1}},\mu}(\Om \cap B_{\frac{1}{n}}(x)) + \frac{1}{n} \leq \E_{{\bf{0}},\mu}(\Om \cap B_{\frac{1}{n}}(x)) + 1 = \E_{{\bf{0}},\mu}(\Om) + 1.
\end{align}
The last equality follows from the fact that $\E_{{\bf{0}},\mu}(.) $ does not depend on the domain.
From \eqref{for_constant}, we have
\begin{eqnarray} \label{ab}
 \E_{{\bf{0}},\mu}(\Om) - \mu  \int_{\Om}  |v_n|^p \dx < \E_{{\bf{1}},\mu}(\Om \cap B_{\frac{1}{n}}(x)) + \frac{1}{n}.
\end{eqnarray}
Since $\mu \in (0,\la_1)$, it follows from \eqref{for_constant} that $(v_n)$ is bounded in $\Dp$. Further, their supports are shrinking to a null set, namely $\{x\}$ as $n \ra \infty$. This implies, $v_n \wra 0$ in $\Dp$, and the Rellich-Kondrachov compactness theorem assures that $v_n \ra 0$ in $L^p(\Om).$ Now, by taking $n \ra \infty$
in the above inequality \eqref{ab}, we obtain $\E_{{\bf{0}},\mu}(\Om) \leq \C_{{\bf{1}},\mu}(x)$, for each $x \in \overline{\Om}$. This implies $\E_{{\bf{0}},\mu}(\Om) \leq \C_{{\bf{1}},\mu}^{*}(\Om)$. 
Further, for all $\mu \in (0,\la_1)$, we  have $\E_{{\bf{1}},\mu}(\Om) < \E_{{\bf{0}},\mu}(\Om)$ when $N \geq p^2$, see \cite[Lemma 7.1]{AlonsoEstimates}. Hence, for all $\mu \in (0,\la_1)$, ${\bf{1}}$ is sub-critical in $\Om$ when $N \geq p^2$. On the other hand, since $\Om$ is bounded, ${\bf{1}}$ is sub-critical at infinity too. 
\end{proof}

\begin{rmk} \rm
By the above proposition and Theorem \ref{mainthm}, we conclude that \eqref{Critical} with $g \equiv 1$ on a bounded domain in $\R^N$ admits a positive solution for all $\mu \in (0,\la_1)$, provided $N \geq p^2$. Thus the results of Br\'ezis-Nirenberg \cite[Theorem 1.1]{Nirenberg} (for $p=2$) and \cite[Theorem 7.4]{AlonsoEstimates} (for general $p$) follows as a particular case of Theorem \ref{mainthm}.
\end{rmk}

 \begin{example}  \label{criatptsandinf}
 \rm We provide some examples of $g$ that are not subcritical.

 $(i)$ {\bf{A potential critical in $\Om$}}. Let $\Om$ be a domain in $\R^N$. Consider the zero function ${\bf{0}}(z)\equiv 0$ on $\Om.$ Fix $\mu \in  \left(0,\mu_1(0)\right)$.  Recall that
$$\E_{{\bf 0},\mu}(\Om)=\displaystyle\inf_{u \in \mathcal{D}_p(\Om)} \left\{\int_{\Om} |\nabla u|^p \dx: \|u\|_{p^{*}}=1 \right\}.$$
It is not difficult to see that $\E_{{\bf 0},\mu}(.)$ is independent of domain (as $\E_{{\bf 0},\mu}(.)$ is invariant under translations and dilations in $\R^N$), and hence $\C_{{\bf 0},\mu}(x)=\E_{{\bf 0},\mu}(\Om)$ for any $x \in \overline{\Om}$. This implies $\C_{{\bf 0},\mu}^*(\Om)=\E_{{\bf 0},\mu}(\Om)$. Therefore, ${\bf{0}}$ is critical in $\Om$.
 
$(ii)$ {\bf{A potential critical in $\Om$}}. Let $\Om=\Om_k \times \Om_{N-k}$ be as in \eqref{domain} with  $0\in \Om$. Let $g(z)=\frac{1}{|z|^p}$; $z \in \Om$. Fix a $\mu \in  \R$.  Recall that
$$\E_{\frac{1}{|z|^p},\mu}(\Om)=\displaystyle\inf_{u \in \mathcal{D}_p(\Om)} \left\{\int_{\Om} [|\nabla u|^p-\frac{\mu}{|z|^p}] \ \dz: \|u\|_{p^{*}}=1 \right\}.$$
Since $\E_{\frac{1}{|z|^p},\mu}(\Om)$ is invariant under dilation, using the scaling arguments it can be seen that $\C_{\frac{1}{|z|^p},\mu}(0)=\E_{\frac{1}{|z|^p},\mu}(\Om)$. This implies $\C_{\frac{1}{|z|^p},\mu}^*(\Om)=\E_{\frac{1}{|z|^p},\mu}(\Om)$. Therefore, $\frac{1}{|z|^p}$ is critical in $\Om$.

$(iii)$ {\bf{A potential critical at infinity}}. Let $\Om=\Om_k \times \Om_{N-k}$ be as in \eqref{domain} with $k=N,$ $a>0$, and $b=\infty$ i.e., the exterior of a ball $B_a(0)$ in $\R^N$. Consider $g(z)=\frac{1}{|z|^p}$; $z \in \Om$. For each $\epsilon >0$, there exists $w \in \mathbb{S}_p(\Om)$ such that 
 $$ \int_{\Om} [|\nabla w|^p(z)-\mu \frac{|w|^p(z)}{|z|^p}] \ \dz < \E_{\frac{1}{|z|^p},\mu}(\Om) + \epsilon.$$
For $w_R(z)= R^{\frac{p-N}{p}}w(\frac{z}{R})$, one can see that $w_R \in \mathbb{S}_p(\Om \cap B_{aR}^c)$ for $R>1$, and
$$\int_{\Om} [|\nabla w_R|^p(z)-\mu \frac{|w_R|^p(z)}{|z|^p}] \dz=\int_{\Om} [|\nabla w|^p(z)-\mu \frac{|w|^p(z)}{|z|^p}] \dz.$$
This gives $\E_{\frac{1}{|z|^p},\mu}(\Om \cap B_{aR}^c) \leq \E_{\frac{1}{|z|^p},\mu}(\Om)$ for all $R>1$ and hence $\C_{\frac{1}{|z|^p},\mu}(\infty) \leq \E_{\frac{1}{|z|^p},\mu}(\Om)$, and the other way inequality always holds. Thus, $g$ is critical at infinity.
\end{example}

Next, we are going to prove Theorem \ref{perfect}.
\begin{rmk} \label{Req4} \rm
Recall that, for any $\mu \in \R,$ $\E_{0,\mu}(.)$ is invariant under dilation and translation in $\R^N$. 
Hence, $\E_{0,\mu}(\Om)=\C_{0,\mu}^*(\Om)$ i.e., $g \equiv 0$ is critical in $\Om$. In addition, if $\Om$ is bounded and star shaped, using the Pohozev identity, we can show that $\E_{0,\mu}(\Om)$ is not attained in $\Dp$. However, for $\Om=\R^N,$ $\E_{0,\mu}(\R^N)$ is attained by the following radial functions in $\mathcal{D}_p(\R^N)$:
$$\Psi_{\epsilon,x_0}(x)= \epsilon^{-\frac{N}{p^*}} \left(1 + \text{C}(N,p)\left|\frac{x-x_0}{\epsilon}\right|^{\frac{p}{p-1}} \right)^{\frac{p-N}{p}},$$
for any $\epsilon >0$ and $x_0 \in \R^N$ \cite[Corollary I.1]{Lions2a}.
\end{rmk}
Let $g \in \mathcal{F}_p(\Om)$ be such that $g^+ \neq 0$. Then there exists a compact set $\text{K} \subseteq \Om$ with $|\text{K}|>0$ such that $g$ is positive on $\text{K}.$  Furthermore, due to Lusin's theorem, we can assume that $g$ is continuous on $\text{K}$. Thus, \begin{align} \label{g_min}
    g_{min}:=\min \{g(x): x \in \text{K}\} >0.
\end{align} Let $\Phi_{\text{K}} \in \text{C}_c^{\infty}(\Om)$ such that $\Phi_{\text{K}} =1$ on $\text{K}$. For each $y \in \R^N$, we consider $$u_{\epsilon,y}= \Phi_{\text{K}} U_{\epsilon,y}$$ where $U_{\epsilon,y}(x) = \chi_{\epsilon,y}(x)\left(\epsilon + \mbox{C}(N,p) |x-y|^{\frac{p}{p-1}} \right)^{\frac{p-N}{p}}$ and  $\chi_{\epsilon,y}(x) =\chi_{0} \left(\frac{x-y}{\epsilon}\right)$ with
$\chi_{0} \in \text{C}_c^{\infty}(\Om)$ is such that $0 \leq \chi_{0} \leq 1$, $\chi_0 =1$ on $B_1(0)$ and vanishes outside $B_2(0)$.
Next, we list some properties of $u_{\epsilon,y}$.
\begin{lemma} \label{estimate}
Let $g \in \H_p(\Om)$ be such that $g^+ \neq 0.$ For each $\epsilon >0$, $u_{\epsilon,y}$ satisfies the following properties:
 \begin{enumerate}
  \item[(i)] $\|u_{\epsilon,y}\|_{p^*}^p= \|\Psi_{1,y}\|_{p^*}^p \ \epsilon^{\frac{p-N}{p}} + O(1) \,,$
  \item[(ii)] $\|\nabla u_{\epsilon,y} \|_{p}^p= \|\nabla \Psi_{1,y}\|_{p}^p \ \epsilon^{\frac{p-N}{p}}  + O(1)\,,$ 
  \item[(iii)] for $y \in {\rm{K}}$, we have $$\displaystyle \int_{\Om} g |u_{\epsilon,y}|^p \dx \geq \begin{cases}
      \text{A} \ g_{min} \ \epsilon^{\frac{p^2-N}{p}}  + O(1) \quad \ \ \mbox{if} \ p^2 <N \\
      \text{A} \ g_{min} \ |\log (\epsilon)| + O(1) \quad \mbox{if} \ p^2 =N,
    \end{cases} $$ 
 \end{enumerate}     where $A>0$ depends only on $N,p$, and $g_{min}$ is defined as \eqref{g_min}.
\end{lemma}
\begin{proof}
For a proof of $(i)$ and $(ii)$, we refer to the assertion (7.7) of \cite{AlonsoEstimates} (one can also see \cite[assertion 1.11, 1.12]{Nirenberg} for $p=2$).  To prove $(iii)$, we first recall the following estimate \cite[(c) of 7.7]{AlonsoEstimates} (one can also see \cite[1.13]{Nirenberg} for $p=2$): 
 $$\| U_{\epsilon,0}^p\|_{1}= \begin{cases*}
                                A \epsilon^{\frac{p^2-N}{p}} + O(1), \ \quad     \mbox{ if} \ p^2  <  N,\\
                                A |\log (\epsilon)| + O(1), \quad \mbox{if} \ p^2 = N
                                \end{cases*} $$
$A$ is a positive constant independent of $\epsilon.$ Observe that $u_{\epsilon,y}(x)= \Phi_{\text{K}}(x) U_{\epsilon,y}(x)=\Phi_{\text{K}}(x) U_{\epsilon,0}(x-y)$. By applying $\frac{U_{\epsilon,0}^{p}}{\|U_{\epsilon,0}^p\|_{1}}$ as an approximate identity, we compute the following:
 \begin{eqnarray*}
  \int_{\Om} g(x) |u_{\epsilon,y}(x)|^p \dx &=& \int_{\Om} g(x) \Phi_{\text{K}}(x)^p U_{\epsilon,0}(x-y)^p \dx  \\
   &=& \|U_{\epsilon,0}^p\|_{1} \int_{\Om} (g \Phi_{\text{K}}^p)(x) \frac{U_{\epsilon,0}^p}{\|U_{\epsilon,0}^p\|_{1}} (x-y) \dx \\
   &=&  \|U_{\epsilon,0}^p\|_{1}  (g\Phi_{\text{K}}^p) (y) + O(1) \\
   & \geq & 
   \begin{cases}
      A \epsilon^{\frac{p^2-N}{p}} g_{min} + O(1) \quad \ \ \  \mbox{if} \ p^2 <N \\
      A |\log (\epsilon)| \ g_{min} + O(1) \quad \mbox{if} \ p^2 =N \,.
    \end{cases}
 \end{eqnarray*}
\end{proof}

\noi {\bf{Proof of Theorem \ref{perfect}:}} Let $g \in \mathcal{F}_p(\Om)$. Proposition \ref{pastthm} assures that $G_p$ is compact on $\Dp.$ Fix $\mu \in (0,\mu_1(g))$ and choose $\epsilon >0$. By definition of $\C_{g,\mu}(x)$, there exists $\delta >0$ such that for each $r \in (0,\delta)$, there exists $u_r \in \mathcal{D}_p(\Om \cap B_r(x))$ with $\|u_r\|_{p^*}=1$ satisfying 
$$\int_{\Om} [|\nabla u_r|^p - \mu g |u_r|^p] \dx < \C_{g,\mu}(x) + \epsilon.$$
This yields,
$$\E_{{\bf 0},\mu}(\Om) - \mu \int_{\Om}  g |u_r|^p \dx  < \C_{g,\mu}(x) + \epsilon.$$
Since $\mu \in (0,\mu_1(g))$, it follows that $(u_r) $ is bounded in $\Dp$. Further, their supports are decreasing to a singletone set $\{x\}$. Hence, $u_r \wra 0$ in $\Dp$ as $r \ra 0$, and hence the compactness of $G_p$ implies $\displaystyle \int_{\Om}  g |u_r|^p \dx \ra 0$ as $r \ra 0$. Thus, $\E_{{\bf 0},\mu}(\Om) \leq \C_{g,\mu}(x)$, for any $x \in \overline{\Om}$. This shows that $\E_{{\bf 0},\mu}(\Om) \leq \C_{g,\mu}^*(\Om)$. 
By the similar reasoning we conclude $ \E_{{\bf 0},\mu}(\Om) \leq \C_{g,\mu}(\infty)$. Thus, in order to show that
$g$ is subcritical in $\Om$ and at infinity, it is enough to show that $\E_{g,\mu}(\Om) < \E_{{\bf 0},\mu}(\Om)$.
To establish this strict inequality, we recall the functions $u_{\epsilon,y} \in \text{C}_c^{\infty}(\Om)$ (for $\epsilon>0$ small) in Lemma \ref{estimate} and get the following estimate:
\begin{eqnarray*}
 Q_g(u_{\epsilon,y}) &:=&  \frac{\displaystyle \int_{\Om} |\nabla u_{\epsilon,y}|^p \dx - \mu \displaystyle \int_{\Om} g |u_{\epsilon,y}|^p \dx}{[\int_{\Om} |u_{\epsilon,y}|^{p^*}\dx]^{\frac{p}{p^*}}} \\
 &\leq&
\begin{cases*}
 \E_{{\bf 0},\mu}(\Om) + O(\epsilon^{\frac{N-p}{p}}) - \text{A} \ g_{min} \ \epsilon^{p-1}, \quad \quad \quad \quad \  \mbox{if} \ p^2 <N \\
 \E_{{\bf 0},\mu}(\Om) + O(\epsilon^{\frac{N-p}{p}}) - \text{A} \ g_{min} \ \epsilon^{\frac{N-p}{p}}  |\log (\epsilon)|, \quad \mbox{if} \ p^2 =N,
\end{cases*}
\end{eqnarray*}
This implies $Q_g(u_{\epsilon,y}) < \E_{{\bf 0},\mu}(\Om)$ for sufficiently small $\epsilon$. By taking $w_{\epsilon}= \frac{u_{\epsilon,y}}{\|u_{\epsilon,y}\|_{p*}}$ we have 
$$\E_{g,\mu}(\Om) \leq \int_{\Om} |\nabla w_{\epsilon}|^p \dx - \mu \int_{\Om} g |w_{\epsilon}|^p \dx = Q_g(u_{\epsilon,y}) < \E_{{\bf 0},\mu}(\Om).$$
This completes our proof.

In the following remark, we exhibit certain classical spaces that are contained in $\F_p(\Om).$
\begin{rmk} \label{Subspace_Fp}
\rm
For $p=2$ and $\Om$ bounded,    $L^r(\Om)\subseteq \F_p(\Om)$ with $r>\frac{N}{2}$  \cite{Manes-Micheletti}, $r=\frac{N}{2}$ \cite{Allegretto}. For $p\in (1,\infty)$ and for general domain $\Om,$ $ L^{\frac{N}{p},d}(\Om) \subseteq \F_p(\Om)$ with $d<\infty$, in \cite{Visciglia}. Furthermore, a larger space $\overline{\text{C}_c^{\infty}(\Om)}$ in $L^{\frac{N}{p},\infty}(\Om)$ is contained in $\F_p(\Om)$ \cite[for $p=2$]{AMM} and \cite[for $p\in (1,N)$]{anoop-p}. For $g\in L^1_{loc}(\Om)$, we consider \[\tilde{g}(r)= {\rm ess}\sup \{|g(y)|: |y|=r \}, \ r > 0,\]
  where the essential supremum is taken with respect to $(N-1)$ dimensional surface measure. 
  Let $I_p(\Om)= \big\{ g\in L^1_{loc}(\Om) : \tilde{g} \in L^{1}((0, \infty),r^{p-1}\dr) \big\}$. In \cite{ADS-exterior}, authors showed that $ I_p(\overline{B}_1^c)\subseteq \F_p(\overline{B}_1^c)$ for $p \in (1,N)$.
\end{rmk}

\section{Critical Potentials} \label{Cric_poten}
In this section, we prove Theorem \ref{symmetricsol} and Theorem \ref{atorigin}. First we observe that, for a Hardy potential $g \in \mathcal{H}_p(\Om)$ which is not sub-critical in $\Om$ or  at infinity, one of the following cases occur: 
\begin{enumerate}[(i)]
    \item $g$ is $H$-subcritical in $\Om$ as well at infinity,
    \item $g$ is $H$-critical in $\Om$ but $H$-subcritical at infinity,
     \item $g$ is $H$-critical at infinity but $H$-subcritical in $\Om$,
     \item $g$ is $H$-critical in $\Om$ as well as at infinity,
\end{enumerate}
for a closed subgroup $H$ of $\mathcal{O}(N).$  Theorem \ref{trivial} deals with the case (i), while Theorem \ref{symmetricsol} and Theorem \ref{atorigin} address the rest of the three cases.  

\noi {\bf{Proof of Theorem \ref{trivial} :}} Let $H$, $\Om$, and $g$ satisfy \ref{H1} and \ref{H2}. If $g$ is $H$-subcritical in $\Om$ and at infinity, then one can repeat similar arguments as in Theorem \ref{mainthm} to show that $\E_{g,\mu}^H(\Om)$ is attained in $ \Dp^H$. Further, the principle of symmetric criticality and strong maximum principle can be applied to ensure that 
\eqref{Critical} admits a positive solution.

Next we prove Theorem \ref{symmetricsol}.

\noi {\bf{Proof of Theorem \ref{symmetricsol} :}}  Let $H$, $\Om$, and $g$ satisfy \ref{H1} and \ref{H2}.
Let $(u_n)$ be a minimising sequence
of $\E_{g,\mu}^{H}(\Om)$ i.e., $ u_n \in \mathcal{D}_{p}(\Om)^{H}$ with $\|u_n\|_{p^*}=1$ and
$$\int_{\Om} [|\nabla u_n|^p - \mu g |u_n|^p] \dx \ra \E_{g,\mu}^{H}(\Om), \ \mbox{as} \ n \ra \infty.$$
Since $\mu \in (0,\mu_1(g))$, we have $(u_n)$ is bounded in
$\mathcal{D}_{p}(\Om)^{H}.$ Hence, $u_n \wra u$ in $\mathcal{D}_{p}(\Om)^{H}$ (upto a subsequence). Now, we use Corollary \ref{Cor_ConCpct} to obtain
\begin{eqnarray}
 \E_{g,\mu}^{H}(\Om) &=& \lim_{n \ra \infty} \int_{\Om} [|\nabla u_n|^p - \mu g |u_n|^p] \dx  \nonumber \\
              &\geq&  \int_{\Om} [|\nabla u|^p - \mu g |u|^p] \dx + \C_{g,\mu}^{H,*}(\Om) \|\nu \|^{\frac{p}{p^*}} + \C_{g,\mu}^{H}(\infty) \nu_{\infty}^{\frac{p}{p^*}} \nonumber \\
             & \geq & \E_{g,\mu}^{H}(\Om)  \left(\int_{\Om} |u|^{p^*} \dx \right)^{\frac{p}{p^*}} + \C_{g,\mu}^{H,*}(\Om) \|\nu \|^{\frac{p}{p^*}} + \C_{g,\mu}^{H}(\infty) \nu_{\infty}^{\frac{p}{p^*}} \nonumber \\
              & \geq & \E_{g,\mu}^{H}(\Om)  \left(\int_{\Om} |u|^{p^*} \dx +  \|\nu \| +  \nu_{\infty} \right)^{\frac{p}{p^*}}=\E_{g,\mu}^{H}(\Om) \label{Req5} 
             \end{eqnarray}
Thus, equality occurs in all the above inequalities. As $g$ is $H$-subcritical at infinity, by the same arguments as in the proof of Theorem \ref{mainthm} we infer that $\nu_{\infty}=0.$ In the view of Corollary \ref{Cor_ConCpct}, observe that the equality in \eqref{Req5} implies that  $\E_{g,\mu}^{H}(\Om) \norm{\nu}^{\frac{p}{p^*}} + \mu \norm{\gamma} = \norm{\Gamma_{\overline{\sum_g} \cup \mathbb{J}}} $. Thus, by Remark \ref{bhul1}-$(i)$, we have $  \E_{g,\mu}^{H}(\Om) \norm{\nu}^{\frac{p}{p^*}} = \norm{\zeta_\mu} $, where $\zeta_\mu$ is as in Proposition \ref{ConCpct:new}. Notice that, the equality in \eqref{Req5} and Corollary \ref{Cor_ConCpct}-$(e)$  implies that  
$$\left[ \displaystyle \left(\int_{\Om} |u|^{p^*} \dx \right)^{\frac{p}{p^*}} + \|\nu \|^{\frac{p}{p^*}} + \nu_{\infty} ^{\frac{p}{p^*}} \right]
    = \displaystyle \left[ \left(\int_{\Om} |u|^{p^*} \dx \right) + \|\nu \|+\nu_{\infty} \right]^{\frac{p}{p^*}}=1.$$  Hence, one among $\|u\|_{p^*},$ $\|\nu \|$, and $\nu_{\infty}$ is $1$, others are  $0$. Since $\nu_{\infty}=0$,   either $\|u\|_{p^*}=1$ and $\|\nu \|=0,$ or $\|u\|_{p^*}=0$ and $\|\nu \|=1$.
Suppose $\|\nu\|=1$ and $\|u\|_{p^*}=0.$
Thus, by Corollary \ref{Cor_ConCpct}-$(i)$ we infer that $\nu $ is concentrated on a finite $H$ orbit. But, by our hypothesis, $Hx$ is infinite for all $x \in \overline{\Om}$ i.e., no points of $\overline{\Om}$ has finite $H$ orbit. This leads to a contradiction.
Therefore, $\nu=0$ and $\|u\|_{p^*}=1$. Now using the arguments as given in the proof of Theorem \ref{mainthm}, we conclude that  $\E_{g,\mu}^{H}(\Om)$ is attained by $u\in\Dp^H$.  Now by the principle of symmetric criticality theory and strong maximum principle, we establish the existence of a positive solution to the problem \eqref{Critical}.

\begin{rmk} \label{symm} \rm
$(i)$ If $\Om$, $g,H$ are as in Theorem \ref{symmetricsol},  then the above proof shows that $\E_{g,\mu}^{H}(\Om)$ is attained at some $u \in \mathcal{D}_p(\Om)^{H}$
irrespective of $\E_{g,\mu}(\Om)$ being attained in $\Dp$. For example, take $g(z)= \frac{1}{|z|^p}$ on $\mathcal{A}:= \displaystyle B_R(0) \setminus B_r(0)$ for $0<r<R<\infty$. Theorem 2.1 of \cite{Ruiz} shows that, for $\mu \in (0,\mu_1(g))$, $\E_{\frac{1}{|z|^p},\mu}(\Om)$ is attained only if $\Om=\R^N$. This infers that $\E_{\frac{1}{|z|^p},\mu}(\mathcal{A})$ is not achieved. However, $g$ is $H$-invariant, in fact, $\mathcal{O}(N)$-invariant, and also $g$ is $H$-subcritical at infinity (as $\mathcal{A}$ is bounded). Therefore, by Theorem \ref{symmetricsol},  $\E_{g,\mu}^{H}(\mathcal{A})$ must be attained.

$(ii)$ Notice that, the solution obtained in Theorem \ref{symmetricsol} is always $H$-invariant. Thus, in particular, if $H=\mathcal{O}(N)$ then the solution (if exists) will be radial.
\end{rmk}


\begin{example}  \label{criticalatpts}
\rm
$(i)$ {\bf{A potential critical in $\Om$}}: Let $\Om=\Om_k \times \Om_{N-k}$ be as in \eqref{domain} and $g(z)=\frac{1}{|y|^p}$; $z=(x,y) \in \Om$. It is known that $g \in \mathcal{H}_p(\Om)$ if $N-k \geq 2$ and $p<N-k$, see \cite[Theorem 2.1]{Tarantello}. We show that $g$ is critical in $\Om$. For a fix $\mu \in (0,\mu_1(g))$ and $\epsilon >0$, there exists $w \in \mbox{C}_c^{\infty}(\Om)$ (by density of $\mbox{C}_c^{\infty}(\Om)$ in $\mathcal{D}_{p}(\Om)$) with $\|w\|_{p^*}=1$ such that 
$$ \int_{\Om} [|\nabla w|^p(z)-\mu \frac{|w|^p(z)}{|y|^p}] \dz < \E_{g,\mu}(\Om) + \epsilon.$$
By taking $(\xi,0) \in \Om_k\times \Om_{N-k}$ and 
$$ w_r(z)= \displaystyle r^{\frac{p-N}{p}} w \left(\frac{x-\xi}{r},\frac{y}{r}\right),$$ on $\Om_r:=\{(x,y) \in \R^k \times \Om_{N-k}: ar \leq |x-\xi| \leq \overline{b}r \}$, where $\overline{b} >a$ is such that $w(x,y)=0$, $ \forall \,\ |x| > \overline{b}$. It is clear that $\|w_r\|_{p^*}=1$ and  $\Om_r \subseteq \Om$ for small $r>0.$ Now, using the change of variable $\frac{x-\xi}{r} =x'$ and $\frac{y}{r}=y'$ we obtain
$$\int_{\Om_r} [|\nabla w_r|^p(z)-\mu \frac{|w_r|^p(z)}{|y|^p}] \dz=\int_{\Om} [|\nabla w|^p(z)-\mu \frac{|w|^p(z)}{|y|^p}] \dz.$$
This gives $\C_{g,\mu}((\xi,0)) \leq \E_{g,\mu}(\Om)$. Consequently, $\C_{g,\mu}^{*}(\Om) \leq \E_{g,\mu}(\Om)$ holds for all $\mu \in (0,\mu_1(g))$, and the other way inequality indeed holds. Therefore, $g$ is critical in $\Om$.

$(ii)$ {\bf{A potential $H$-subcritical at infinity in a unbounded domain}}:  Consider the same example as above where $\Om=\Om_k \times \Om_{N-k}$ as in \eqref{domain} with $0<a, b = \infty$ and $ N-k \geq 2, p < N-k$.
Then it follows from the similar arguments used in Example \ref{criatptsandinf}-$(iii)$ that $g$ is critical at infinity. On the other hand, if $b < \infty$, then $\Om_k$ is bounded, and on top of that if $\Om_{N-k}$ is bounded, then $\Om$ becomes bounded. In that case $g$ is $H$-subcritical at infinity for any subgroup $H$ of $\mathcal{O}(N)$. Next consider that $b<\infty$ and $\Om_{N-k}=\R^{N-k}$. In this case, we show that $g$ is $H_{N-k}$-subcritical at infinity, where $H_{N-k}=  \{Id_{k}\}\times\mathcal{O}(N-k)$.   Fix $\mu \in (0,\mu_1(g))$. For each $R>0$, by definition of $\C_{g,\mu}^{H_{N-k}}(\infty)$, there exists $v_R \in \mathbb{S}_p(\Om \cap B_R^c)^{H_{N-k}}$ such that
$$\int_{\Om} [|\nabla v_R|^p(x,y) - \mu \frac{|v_R|^p(x,y)}{|y|^p}] \dz < \C_{g,\mu}^{H_{N-k}}(\infty) + \frac{1}{R}.$$
Therefore,
\begin{eqnarray} \label{1015}
 \E_{{\bf{0}},\mu}^{H_{N-k}}(\Om) - \mu  \int_{\Om_R}  \frac{|v_R|^p(x,y)}{|y|^p}\dz  < \C_{g,\mu}^{H_{N-k}}(\infty) + \frac{1}{R},
\end{eqnarray}
where $\Om_R=\Om \cap B_R^c$.
Notice that $\mu \in (0,\mu_1(g))$ implies that $(v_R)$ is bounded in $\mathcal{D}_{p}(\Om_R)$ and their supports are decreasing to infinity. Hence, $v_R \wra 0$ in $\mathcal{D}_{p}(\Om)$. Now, since $g \in L^{\frac{N}{p}}(\Om_R)$, it follows that $g \in \mathcal{F}_p(\Om_R)$ (Remark \ref{Subspace_Fp}) and hence
$$\displaystyle \int_{\Om_R}  \frac{|v_R|^p(x,y)}{|y|^p} \dz \ra 0\,, \ 
\mbox{as} \ R \ra \infty.$$ 
By taking $R \ra \infty$ in \eqref{1015}, we obtain $\E_{{\bf{0}},\mu}^{H_{N-k}}(\Om) \leq \C_{g,\mu}^{H_{N-k}}(\infty)$. 
Now, in order to show that $g$ is $H_{N-k}$-subcritical at infinity, we require to show that $  \E_{g,\mu}^{H_{N-k}}(\Om)< \E_{{\bf{0}},\mu}^{H_{N-k}}(\Om)$. To show this, we recall the well known compact embedding proved by Lions \cite[LEMME III.2]{LIONS1982315} that ensures  $\Dp^{H_{N-k}} \hookrightarrow L^{p^*}(\Om)$ is compact. As a consequence $\E_{{\bf{0}},\mu}^{H_{N-k}}(\Om)$ is achieved at some $u \in \mathbb{S}_p(\Om)^{H_{N-k}}$  
and consequently,
$$\E_{{\bf{0}},\mu}^{H_{N-k}}(\Om)= \int_{\Om} |\nabla u|^p \dz > \int_{\Om} [|\nabla u|^p-\mu g |u|^p] \dz \geq \E_{g,\mu}^{H_{N-k}}(\Om)\,. $$

$(iii)$ {\bf{A potential critical at infinity but $H$-subcritical at infinity}}: Let $\Om=\Om_k \times \Om_{N-k}$ as in \eqref{domain} with $0<a, b < \infty$ and $\Om_{N-k}= \R^{N-k}$. Also, assume that $ N-k \geq 2, p < N-k$. Consider $g(z)=\frac{1}{|z|^p}$; $z=(x,y) \in \Om$. One can repeat the arguments of Example \ref{criticalatpts}-$(ii)$ to show that $g$ is $H_{N-k}$-subcritical at infinity.  Now, we show that $g$ is critical at infinity. On the contrary, if $g$ is subcritical at infinity, then it follows from the proof of Theorem \ref{symmetricsol} that $\E_{g,\mu}(\Om)$ is achieved. However, this is possible only if $\Om=\R^N$ \cite[Theorem 2.2]{Ruiz}. Hence, $g$ must be critical at infinity.

$(iv)$ {\bf{A potential $H$-subcritical at infinity}}: Let $\Om=\Om_k \times \R^{N-k}$ be as in \eqref{domain} with $b <\infty$ and $ N-k \geq 2, p < N-k$. For $z=(x,y) \in \Om$, let
$$ g(z)=\frac{1}{|x|^{\al} (1+|y|^2)^{\frac{p-\al}{2}}} \,, \ \al \in (0, \frac{pk}{N}) \,. $$
Using the same arguments as in the previous example, one can show that $g$ is $H_{N-k}$-subcritical at infinity, where $H_{N-k}= \{Id_{k}\}\times \mathcal{O}(N-k)$. 
\end{example}
\begin{rmk}\rm 
Let $\Om=\Om_k \times \Om_{N-k}$ be as in \eqref{domain} with $0<a, b<\infty$ and $g$ be as in $(ii)$, $(iii)$ in Example \ref{criticalatpts}. Then $g$ is $H_{N-k}$ invariant and $H_{N-k}$-subcritical at infinity for certain range of $k$, where $H_{N-k}= \{Id_{k}\}\times \mathcal{O}(N-k)$.
In these cases, Theorem \ref{symmetricsol} can be applied to show that \eqref{Critical}
admits a positive solution. 
\end{rmk}

Next we prove Theorem \ref{atorigin}.

\noi {\bf{Proof of Theorem \ref{atorigin}:}} $(i)$ Let $\Om = \R^N$ and $H$ be a closed subgroup of $\mathcal{O}(N)$ that acts on $\R^N$.
Let $(u_n)$ be a minimising sequence
of $\E_{g,\mu}^H(\R^N)$ on $\mathbb{S}_p(\R^N)^H$ i.e., $ u_n \in \mathcal{D}_{p}(\R^N)^H$ with $\|u_n\|_{p^*}=1$ and
$$\int_{\R^N} [|\nabla u_n|^p - \mu g |u_n|^p] \dx \ra \E_{g,\mu}^H(\R^N), \ \mbox{as} \ n\ra \infty.$$
For each $u_n$, there exists $R_n>0$ such that 
$$\int_{ B_{R_n}} |u_n|^{p^*} \dx \geq \frac{1}{2} \,.$$
We define $w_n(z)= R_n^{\frac{N-p}{p}} u_n(R_nz)$ on $\R^N$. 
Then, $w_n \in \mathcal{D}_{p}(\R^N)^H$ and $\|w_n\|_{p^*}=1$.  
Also, since $g$ satisfies \ref{H3} for small $r>0$, we have
$$ \E_{g,\mu}^H(\R^N) \leq \int_{\R^N} [|\nabla w_n|^p - \mu g |w_n|^p]\dz \leq \int_{\R^N} [|\nabla u_n|^p - \mu g |u_n|^p] \dx \ra \E_{g,\mu}^H(\R^N), \ \ \mbox{as} \ n \ra \infty.$$
$\mu \in (0,\mu_1(g))$ ensures that $(w_n)$ is bounded in
$\mathcal{D}_{p}(\R^N)^H$, and hence $w_n \wra w$ in $\mathcal{D}_{p}(\R^N)^H$. 
Now, we use Corollary \ref{Cor_ConCpct} to obtain,
\begin{eqnarray*}
 \E_{g,\mu}^H(\R^N) &=& \lim_{n \ra \infty} \int_{\R^N} [|\nabla w_n|^p - \mu g |w_n|^p] \dz \nonumber\\
& \geq & \int_{\R^N} [|\nabla w|^p - \mu g |w|^p] \dz + \C_{g,\mu}^{H,*}(\R^N) \|\nu \|^{\frac{p}{p^*}} + \C_{g,\mu}^{H}(\infty) \nu_{\infty}^{\frac{p}{p^*}} \nonumber \\
& \geq & \E_{g,\mu}^H(\R^N)  \left(\int_{\R^N} |w|^{p^*} \dz \right)^{\frac{p}{p^*}} + \C_{g,\mu}^{H,*}(\R^N) \|\nu \|^{\frac{p}{p^*}} + \C_{g,\mu}^{H}(\infty) \nu_{\infty}^{\frac{p}{p^*}}\\
&\geq& \E_{g,\mu}^H(\R^N) \left[\int_{\R^N} |w|^{p^*} \dz + \|\nu\|+ \nu_{\infty} \right]^{\frac{p}{p^*}} \geq \E_{g,\mu}^H(\R^N)
\end{eqnarray*}
Thus, equality occurs in $\left[ \displaystyle \left(\int_{\R^N} |w|^{p^*} \dz \right)^{\frac{p}{p^*}} + \|\nu \|^{\frac{p}{p^*}} + \nu_{\infty}^{\frac{p}{p^*}} \right]
    = \displaystyle \left[ \int_{\R^N} |w|^{p^*} \dz + \|\nu \| + \nu_{\infty} \right]^{\frac{p}{p^*}}.$
Hence, exactly one of $\|w\|_{p^*},$ $\|\nu \|$ or $\nu_{\infty}$ is $1$, others are $0.$ Since
$$\int_{B_1(0)} |w_n(z)|^{p^*} \dz= \int_{B_{R_n}(0)} |u_n(x)|^{p^*} \dx \geq \frac{1}{2}  \,, \ \forall n \in \N \,,$$
it follows that $\nu_{\infty} =0$. Since $g$ is $H$-subcritical in $\R^N$, we have $\E_{g,\mu}^H(\R^N) <\C_{g,\mu}^{H,*}(\R^N)$. Hence, one can use the arguments in the proof of Theorem \ref{mainthm} to conclude $\|\nu\|=0$. Therefore, $\|w\|_{p^*}=1$, and again using the arguments as in Theorem \ref{mainthm}, we infer that \eqref{Critical} admits a positive solution in $\R^N$.

$(ii)$ 
We consider $H=\mathcal{O}(N)$ and $\Om = \R^N$. Let $(u_n)$ be a minimising sequence
of $\E_{g,\mu}^H(\R^N)$ on $\mathbb{S}_p(\R^N)^H$ i.e., $ u \in \mathcal{D}_{p}(\R^N)^H$ with $\|u_n\|_{p^*}=1$ and
$$\int_{\R^N} [|\nabla u_n|^p - \mu g |u_n|^p] \dx \ra \E_{g,\mu}^H(\R^N), \ \mbox{as} \ n\ra \infty.$$
Now, for each $n \in \N$, there exists $R_n>0$ such that 
$$\int_{B_{R_n}(0)} |u_n|^{p^*} \dx = \frac{1}{2}.$$
We define $w_n(z)= R_n^{\frac{N-p}{p}} u_n(R_nz)$ on $\R^N$. Then, $w_n \in \mathcal{D}_{p}(\R^N)^H$ with $\|w_n\|_{p^*}=1$ and
\begin{equation} \label{Req11}
  \int_{B_1(0)} |w_n|^{p^*} \dz = \frac{1}{2}  
\end{equation}
Further, since $g$ satisfies \ref{H3} for all $r>0$, we obtain
$$\E_{g,\mu}^H(\R^N) \leq \int_{\R^N} [|\nabla w_n|^p - \mu g |w_n|^p] \dz \leq \int_{\R^N} [|\nabla u_n|^p - \mu g |u_n|^p] \dx \ra \E_{g,\mu}^H(\R^N), \ \ \mbox{as} \ n \ra \infty.$$
Now, $\mu \in (0,\mu_1(g))$ ensures that $(w_n)$ is bounded in
$\mathcal{D}_{p}(\R^N)^H$, and hence $w_n \wra w$ in $\mathcal{D}_{p}(\R^N)^H$. 
Now, we use Corollary \ref{Cor_ConCpct} to obtain,
\begin{eqnarray}
 \E_{g,\mu}^H(\R^N) &=& \lim_{n \ra \infty} \int_{\R^N} [|\nabla w_n|^p - \mu g |w_n|^p] \dz
 \nonumber \\
& \geq & \E_{g,\mu}^H(\R^N)  \left(\int_{\R^N} |w|^{p^*} \dz \right)^{\frac{p}{p^*}} + \C_{g,\mu}^{H,*}(\R^N) \|\nu \|^{\frac{p}{p^*}} + \C_{g,\mu}^{H}(\infty) \nu_{\infty}^{\frac{p}{p^*}} \nonumber \\
&\geq& \E_{g,\mu}^H(\R^N) \left[\int_{\R^N} |w|^{p^*} \dz + \|\nu\|+ \nu_{\infty} \right]^{\frac{p}{p^*}} \geq \E_{g,\mu}^H(\R^N) \label{Req12}
\end{eqnarray}
Thus, equality occurs in \eqref{Req12}. Therefore, as we have seen in the proof of Theorem \ref{symmetricsol} that $  \E_{g,\mu}^{H}(\Om) \norm{\nu}^{\frac{p}{p^*}} = \norm{\zeta_\mu} $, where $\zeta$ is as in Proposition \ref{ConCpct:new}, and also  $$\left[ \displaystyle \left(\int_{\R^N} |w|^{p^*} \dz \right)^{\frac{p}{p^*}} + \|\nu \|^{\frac{p}{p^*}} + \nu_{\infty}^{\frac{p}{p^*}} \right]
    = \displaystyle \left[ \int_{\R^N} |w|^{p^*} \dz + \|\nu \| + \nu_{\infty} \right]^{\frac{p}{p^*}} =1.$$
Hence, exactly one of $\|w\|_{2^*},$ $\|\nu \|$ or $\nu_{\infty}$ is $1$, others are $0.$ By \eqref{Req11}, $\nu_{\infty} \leq \frac{1}{2}$ and hence $\nu_{\infty}=0$.
Now, if $\|\nu\|=1$, then $\|w\|_{p^*}=0$, and since $  \E_{g,\mu}^{H}(\Om) \norm{\nu}^{\frac{p}{p^*}} = \norm{\zeta_\mu} $, it follows from Corollary \ref{Cor_ConCpct}-$(i)$ that, either $\nu =0$ or it is concentrated on a finite $H$ orbit. Since only $0$ has a finite $H$-orbit, it follows that either $\nu=0$ or $\nu=\delta_0$. Let $\nu=\delta_0$. Choose $\phi \in \text{C}_c^{\infty}(\R^N)$ with $0 \leq \phi \leq 1$, $\phi =1$ on $B_{\frac{1}{2}}(0)$ and $\phi =0$ outside $B_1(0)^c$. Using \eqref{Req11} we obtain,
$$\frac{1}{2} = \lim_{n \ra \infty} \int_{B_1(0)} |w_n|^{p^*} \dz \geq \lim_{n \ra \infty} \int_{\R^N} |w_n|^{p^*} \phi \ \dz = \delta_{0} (\phi) =1,$$
which is a contradiction. Thus, $\nu =0$ and hence, $\|w\|_{p^*}=1$.
Now, following the arguments as in Theorem \ref{mainthm}, we infer that \eqref{Critical} admits a positive solution in $\R^N$.


\section{A necessary condition} \label{necessity}
In this section, we prove a necessary condition for $g$ so that \eqref
{Critical} admits a positive solution in entire $\R^N.$ A similar result has been derived in \cite[Theorem 6.1.3]{Anoopthesis} for the existence of solution to the problem:
$$-\De_p u= g(x) |u|^{p-2}u \ \mbox{in} \ \mathcal{D}_p(\R^N).$$
We adapt their ideas  for proving Theorem \ref{ness}. For this we need the following regularity result of solutions.
\begin{proposition} \label{regularity}
Let $g \in {\rm{C}}^{\al}_{loc}(\R^N)$ with $\al \in (0,1)$ and $u$ be a solution to
$$-\De_p u - g |u|^{p-2}u= |u|^{p^*-2}u \ {\rm{in}} \ \mathcal{D}_p(\R^N).$$
Then $u \in {\rm{C}}^{1,\al}_{loc}(\R^N).$
\end{proposition}
\begin{proof}
Let $\Om$ be any bounded domain in $\R^N.$ Then, by following the arguments of Proposition A.1 of \cite{Egneel} (we also refer to \cite[Theorem E.0.20]{Peral}) we can show that $u \in L^{\infty}(\Om)$. Subsequently, using Tolksdrof's regularity results \cite{Tolksdorf} for general quasilinear operator, we have $u \in \text{C}^{1,\al}_{loc}(\R^N).$
\end{proof}

\noi {\bf{Proof of Theorem \ref{ness}:}} Using the above regularity (Proposition \ref{regularity}) we first show that any solution $u$ of \eqref{Critical} satisfies a pointwise identity outside the set where $\nabla u$ vanishes. For each $\eta >0$, we set $\Om_{\eta} = \{x \in \R^N : |\nabla u|> \eta \}$. Now since $-\De_p$ is uniformly elliptic on $\Om_{\eta}$ and $u \in \text{C}^{1,\al}_{loc}(\R^N)$, it follows from standard elliptic regularity theory  \cite[Section 8.3, page 482]{evans} that $u \in \text{C}^{2,\al}_{loc}(\Om_{\eta})$. Thus we obtain the following point-wise identity on $\Om_{\eta}$:
\begin{equation} \label{pointwise}
 -\De_p u - \mu g(x)|u|^{p-2}u= |u|^{p^*-2}u \ \ \mbox{a.e in} \ \Om_{\eta}
\end{equation}
Now we choose a cut-off function $\zeta \in \text{C}_c^{\infty}(\R)$ with $0 \leq \zeta \leq 1 $ such that
$\zeta(t) =1$ for $t \in [0,1]$ and $\zeta =0$ for $t \geq 2.$ For each $n \in \N$, we consider
$$\psi_n(x)= \zeta \left(\frac{|x|^2}{n^2} \right).$$
Then there exists $C>0$ independent of $n$ such that
$$|\psi_n(x)|, |x||\nabla \psi_n(x)| \leq C,$$
for all $x \in \R^N$ and $n \in \N.$ Now multiplying \eqref{pointwise} by $\{x.\nabla u\} \psi_n$, 
\begin{equation} \label{pointwise2}
 -\De_p u \{x.\nabla u\} \psi_n - g(x)|u|^{p-2}u \{x.\nabla u\} \psi_n = |u|^{p^*-2}u \{x.\nabla u\} \psi_n \ \ \mbox{a.e in} \ \Om_{\eta}
\end{equation}
For convenience, we denote
\begin{eqnarray*}
L_n &=& |\nabla u|^{p-2} \nabla u \{x. \nabla u\} \psi_n - (x |\nabla u|^p) \psi_n \\
K_n &=& -\left[ g(x) |u|^{p-2}u + |u|^{p^*-2}u\right] \{x.\nabla u\} \psi_n \\
&+& (1-\frac{N}{p}) |\nabla u|^p \psi_n 
+ \{x. \nabla u\} |\nabla u|^{p-2} \nabla u. \nabla \psi_n.
\end{eqnarray*}
Using \eqref{pointwise2} and following the estimates in \cite[Theorem 6.1.3]{Anoopthesis} we can show that $div L_n = K_n \ \mbox{a.e. in} \ \Om_{\eta}$, and furthermore
$$div L_n = K_n \ \mbox{in} \ \R^N $$
in distribution sense.  Since we have $\text{C}^{1,\al}_{loc}(\R^N)$ regularity of $u$ (Proposition \ref{regularity}), by using weak divergence theorem \cite[Lemma A.1]{Cuesta}, we obtain
 $$\int_{B_{\sqrt{2}n}} K_n(x) \ dx =0.$$
 Furthermore, following the steps of \cite[Theorem 6.1.3]{Anoopthesis} we estimate
\begin{eqnarray*}
div \left \{\left[ g(x) \frac{|u|^p}{p} + \frac{|u|^{p^*}}{p^*} \right]x \psi_n(x) \right\} 
 &=& N \left[ g(x) \frac{|u|^p}{p} + \frac{|u|^{p^*}}{p^*}\right]\psi_n(x) + \frac{|u|^p}{p}  [x. \nabla g(x)] \psi_n(x) \\
 &+& \big[ g(x)|u|^{p-2}u + |u|^{p^*-2}u \big]\{x.\nabla u\} \psi_n \\
 &+& \left[ g(x) \frac{|u|^p}{p} + \frac{|u|^{p^*}}{p^*}\right] x. \nabla \psi_n(x)
\end{eqnarray*}
a.e. in $\R^N.$ Therefore, 
\begin{eqnarray*}
K_n &=& N \left[ g(x) \frac{|u|^p}{p} + \frac{|u|^{p^*}}{p^*} \right] \psi_n(x) + \frac{|u|^p}{p}  [x. \nabla g(x)] \psi_n(x) \\
&+& \left[ g(x) \frac{|u|^p}{p} + \frac{|u|^{p^*}}{p^*}\right] x. \nabla \psi_n(x) - div \left \{[ g(x) \frac{|u|^p}{p}  + \frac{|u|^{p^*}}{p^*}]x \psi_n(x) \right\} \\
&+& (1-\frac{N}{p}) |\nabla u|^p \psi_n 
+ \{x. \nabla u\} |\nabla u|^{p-2} \nabla u. \nabla \psi_n 
\end{eqnarray*}
Hence,
\begin{eqnarray*}
 \int_{B_{\sqrt{2}n}} \left[ N \left( g(x) \frac{|u|^p}{p} + \frac{|u|^{p^*}}{p^*}\right) + \frac{|u|^p}{p}  [x. \nabla g(x)] \psi_n(x) +(1-\frac{N}{p}) |\nabla u|^p \right] \psi_n \\
 \int_{B_{\sqrt{2}n}} \left( \{x. \nabla u\} |\nabla u|^{p-2} \nabla u + \left[ g(x) \frac{|u|^p}{p} + \frac{|u|^{p^*}}{p^*}\right] x \right) \nabla \psi_n(x)=0
\end{eqnarray*}
Notice that $\psi_n \ra 1$ and $\nabla \psi_n \ra 0$ as $n \ra \infty.$
Since each of the above integrals are integrable in entire $\R^N$, we use dominated convergence theorem to obtain
$$\int_{\R^N} \left[ N \left[ g(x) \frac{|u|^p}{p} + \frac{|u|^{p^*}}{p^*}\right] + \frac{|u|^p}{p} [x. \nabla g(x)]  +(1-\frac{N}{p}) |\nabla u|^p \right] =0. \label{Req1}$$
As $u$ is a solution,
$$\int_{\R^N}  |\nabla u|^p = \int_{\R^N} [ g |u|^p + |u|^{p^*}] .$$
Substituting this in \eqref{Req1},
$$\int_{\R^N} [x.\nabla g(x) + p g(x)]|u|^p=0 .$$

\begin{rmk} \label{exterior} \rm
Consider the function 
$$g(z)=\frac{1}{(1+|y|^2)^{\frac{p}{2}}},$$
for $z=(x,y) \in \R^N$.
Then one can verify that $z.\nabla g(z) + p g(z) >0$ in $\R^N.$ Therefore, \eqref{Critical} does not have a solution in entire $\R^N$.
However, if we consider the domain $\Om = \Om_k \times \Om_{N-k}$ as in \eqref{domain} with $0<a<b<\infty$, then following the arguments of  Example \ref{criticalatpts}-B, it can be shown that 
$g$ is sub-critical at infinity if $ N-k \geq 2, p < N-k$. Thus, it follows from Theorem \ref{symmetricsol} that \eqref{Critical} admits a positive solution in $\Om$. 
\end{rmk}

\begin{center}
	{\bf Acknowledgements}
\end{center}

The second author acknowledges the support of the Israel Science Foundation (grant 637/19) founded by the Israel Academy of Sciences and Humanities.


		

\end{document}